\documentclass[12pt, english]{article}
\usepackage{geometry}   
\usepackage{amssymb}
\usepackage{mathrsfs}
\usepackage{amsmath}
\usepackage{amsfonts}
\usepackage{latexsym}
\usepackage{graphicx}
\usepackage{amsthm,enumerate,exscale,stmaryrd}
\usepackage{cases}
\usepackage{bm}
\usepackage[pdftex,bookmarks,colorlinks,breaklinks=true]{hyperref}
\usepackage{color}
\input xy
\xyoption{all}
\usepackage{dsfont}
\usepackage{yfonts}
\usepackage{bbm}
\usepackage{leftidx}
\usepackage{upgreek}
\usepackage[all]{xy}
\usepackage{verbatim}
\usepackage[multiple]{footmisc}
\usepackage{charter}

\NoComputerModernTips
\newdir^{ (}{{}*!/-3pt/\dir^{(}}
\newdir_{ (}{{}*!/-3pt/\dir_{(}}

\newcommand{\Id}{\mathop{\rm Id}\nolimits}

\newcommand{\xqed}[1]{%
  \leavevmode\unskip\penalty9999 \hbox{}\nobreak\hfill
  \quad\hbox{\ensuremath{#1}}}

\newcommand{\ovset}[2]{\overset{#1}{#2}}

\def\Alt{\mathop{\rm Alt}\nolimits}

\def\Sym{\mathop{\rm Sym}\nolimits}
\def\Mult{\mathop{\rm Mult}\nolimits}

\def\Hom{\mathop{\rm Hom}\nolimits}

\def\Ens{\mathop{\bf Ens}\nolimits}

\def\Proj{\mathop{\bf Proj}\nolimits}
\def\Spec{\mathop{\bf Spec}\nolimits}
\def\Spa{\mathop{\bf Spa}\nolimits}

\def\Spf{\mathop{\bf Spf}\nolimits}

\def\Lie{\mathop{\rm Lie}\nolimits}

\def\rank{\mathop{\rm rank}\nolimits}

\def\qlog{{\rm qlog}}
\def\ad{{\rm ad}}

\def\Def{{\rm \emph{Def}}}

\def\innHom{\underline{\Hom}}

\let\phi\varphi
\let\epsilon\varepsilon

\newtheorem{thm}[equation]{Theorem}

\newtheorem{cor}[equation]{Corollary}

\newtheorem{lem}[equation]{Lemma}

\newtheorem{prop}[equation]{Proposition}

\theoremstyle{definition}
\newtheorem{dfn}[equation]{Definition}                
\newtheorem{ex}[equation]{Example}              
\newtheorem{rem}[equation]{Remark}

\newtheorem{cons}[equation]{Construction}
\newtheorem{notation}[equation]{Notations}

\newcounter{example}
\renewcommand{\theexample}{\arabic{example}}

\numberwithin{equation}{section}

\newcommand{\BC}{{\mathbb{C}}}
\newcommand{\BD}{{\mathbb{D}}}

\newcommand{\BG}{{\mathbb{G}}}

\newcommand{\BM}{{\mathbb{M}}}

\newcommand{\BP}{{\mathbb{P}}}
\newcommand{\BQ}{{\mathbb{Q}}}

\newcommand{\BX}{{\mathbb{X}}}

\newcommand{\BZ}{{\mathbb{Z}}}

\newcommand{\Fm}{{\mathfrak{m}}}

\newcommand{\Fp}{{\mathfrak{p}}}

\newcommand{\FL}{{\mathfrak{L}}}

\newcommand{\CB}{{\cal B}}

\newcommand{\CE}{{\cal E}}
\newcommand{\CF}{{\cal F}}
\newcommand{\CG}{{\cal G}}

\newcommand{\CM}{{\cal M}}

\newcommand{\CO}{{\cal O}}
\newcommand{\CP}{{\cal P}}

\newcommand{\CT}{{\cal T}}
\newcommand{\CU}{{\cal U}}
\newcommand{\CV}{{\cal V}}

\newcommand{\lbb}{\llbracket}
\newcommand{\rbb}{\rrbracket}

\newcommand{\ep}{{ \wedge}}

\def\into{\hookrightarrow}
\let\onto\twoheadrightarrow

\newbox\mybox
\def\arrover#1{\mathrel{
       \setbox\mybox=\hbox spread 1em
              {\hfil$\scriptstyle#1\vphantom{g}$\hfil}
       \vbox{\offinterlineskip\copy\mybox
             \hbox to\wd\mybox{\rightarrowfill}}}}             
\def\invlim{\mathop{\vtop{\hbox{\rm lim}\vskip-8pt
        \hbox{\hskip1pt$\scriptstyle\longleftarrow$}\vskip-1pt}}}
\def\dirlim{\mathop{\vtop{\hbox{\rm lim}\vskip-8pt
        \hbox{\hskip1pt$\scriptstyle\longrightarrow$}\vskip-1pt}}}
\def\ontoover#1{\mathrel{
       \setbox\mybox=\hbox spread 1.4em{\hfil$\scriptstyle#1$\hfil}
       \vbox{\offinterlineskip\copy\mybox
             \hbox to\wd\mybox{\rightarrowfill\hskip-2.8mm
                               $\rightarrow$}}}}

\def\to{\rightarrow}

\title{A Cartesian Diagram of Rapoport-Zink Towers over Universal Covers of $p$-Divisible Groups}
\author{S. Mohammad Hadi Hedayatzadeh}

\date{}

\begin{document}
\overfullrule=0pt

\maketitle
\abstract{In \cite{MR3272049}, Scholze and Weinstein show that a certain diagram of perfectoid spaces is Cartesian. In this paper, we generalize their result. This generalization will be used in a forthcoming paper of ours (\cite{HE}) to compute certain non-trivial $\ell$-adic \'etale cohomology classes appearing in the the generic fiber of Lubin-Tate and Rapoprt-Zink towers. We also study the behavior of the vector bundle functor on the fundamental curve in $p$-adic Hodge theory, defined by Fargues-Fontaine, under multilinear morphisms.}
\normalsize
\tableofcontents

\parindent0pt


\section{Introduction}

Let $R$ be a $p$-adically complete $\BZ_{p}$-algebra and let $G$ be a $p$-divisible group over $R$. Let $\tilde{G}$ be the universal cover of $G$. This is the functor on $\text{Nilp}_{R}$ sending an $R$-algebra $S$, on which $p$ is nilpotent, to the inverse limit \[\invlim_{p.}G(S)\] where the transition morphisms are given by multiplication by $p$. When $G$ is connected, this functor, which is a sheaf of $\BQ_{p}$-vector spaces, is represented, under some mild conditions (see \cite[Proposition 3.1.3, Corollary 3.1.5]{MR3272049}) by the formal scheme \[\Spf R\lbb T_{1}^{1/p^{\infty}},\dots,T_{d}^{1/p^{\infty}}\rbb\]

Let $k$ be an algebraically closed field of characteristic $p>0$ and let $H$ be a $p$-divisible group over $k$ of height $h$ and dimension $d$. The universal cover $\tilde{H}$ lifts uniquely to $W(k)$. Let $\CM_{H}^{\infty}$ be the Rapoport-Zink space at infinite level, associated with $H$. This is the rigid analytic space classifying deformations (up to isogeny) of $H$ together with infinite Drinfeld level structure. In \cite{MR3272049}  Scholze and Weinstein show that the adic generic fiber of $\CM_{H}^{\infty}$, denoted by $\CM_{H,\eta}^{\infty,\ad}$, has a natural structure of a perfectoid space. They also prove that as an adic space it embeds canonically inside the $h$-fold product of $\tilde{H}^{\ad}_{\eta}$ and is given by two $p$-adic Hodge theoretic conditions; The first condition defining a point on a certain flag variety canonically attached to $H$ via Grothendieck-Messing deformation theory and the second condition describing the geometric points as certain modifications of vector bundles on the Fargues-Fontaine curve in $p$-adic Hodge theory.\\ 

Now assume that $H$ has dimension $1$. In \cite{HF} and \cite{H2} we have constructed exterior powers of $p$-divisible groups of dimension at most $1$. Using the highest exterior power of $H$, some ad-hoc arguments special to the highest exterior power, and their classification of $p$-divisible groups over $\CO_{\BC_{p}}$ (the ring of integers of $\BC_{p}$, the $p$-adic completion of an algebraic closure of $\BQ_{p}$), Scholze-Weinstein prove that the following diagram is Cartesian

\[\xymatrix{ \CM_{H,\eta}^{\infty,\ad}\ar[r]^{\det}\ar@{^{(}->}[d] & \CM_{\mu_{p^{\infty}},\eta}^{\infty,\ad}\ar@{^{(}->}[d]\\
\big(\tilde{H}^{\ad}_{\eta}\big)^{\times h}\ar[r]_{\det}&\tilde{\mu}^{\ad}_{p^{\infty},\eta}}\]
where the horizontal morphisms are given by the determinant morphisms (suitably defined by taking highest exterior powers), and the vertical morphisms are the embeddings to which we alluded above.\\

In this paper, we investigate the case, where instead of the highest exterior power, we take an arbitrary exterior power of $H$ and generalize Scholze-Weinstein result by proving that the following diagram is Cartesian:

\[\xymatrix{ \CM_{H,\eta}^{\infty,\ad}\ar[r]^{\Lambda^{r}}\ar@{^{(}->}[d] & \CM_{\ep^{r}H,\eta}^{\infty,\ad}\ar@{^{(}->}[d]\\
\big(\tilde{H}^{\ad}_{\eta}\big)^{\times h}\ar[r]_{\Lambda^{r}}&\big(\widetilde{\ep^{r}H}^{\ad}_{\eta}\big)^{\times \binom{h}{r}}}\]

Here the top horizontal morphism is given by using exterior powers of $p$-divisible groups and a result in \cite{HE}, where we show that Drinfeld level structures are preserved under the operation of taking exterior powers and therefore induce a natural morphism from the Lubin-Tate tower to the Rapoport-Zink tower. The lower horizontal morphism is constructed by a careful study of multilinear morphisms of vector bundles on the Fargues-Fontaine curve in $p$-adic Hodge theory, and using results in \cite{MR3272049} and \cite{courbe} relating universal cover a $p$-divisible group $G$ to the global sections of the vector bundle (of slopes between $0$ and $1$) over the Fargues-Fontaine curve associated with $G$.\\

Although we employ some similar arguments as in the case of $r=h$ (using the classification of $p$-divisible groups in terms of vector bundles over Fargues-Fontaine curve), the ad-hoc arguments used by Scholze-Weinstein in the case $r=h$ fail in the general case (note that when $r=h$, there is an isomorphism $\ep^{h}H\cong\mu_{p^{\infty}}$). In the general situation, we had to build a theory of multilinear morphisms of quasi-coherent sheaves over projective schemes and use it to study the behavior of the equivalence of categories, from the category of isocrystals to the category of vector bundles over Fargues-Fontaine curve, under mutilinear morphisms and tensor operations.\\

The main theorem of this paper (that that above diagram is Cartesian) will be used in a forthcoming work (\cite{HE}), where we use the wedge morphism on the Lubin-Tate tower (morphism $\CM^{\infty}_{H}\arrover{\Lambda^{r}}\CM_{\ep^{r}H}^{\infty}$) to study the $\ell$-adic \'etale cohomology of the generic fiber Rapoport-Zink tower. More precisely, in \cite{HE}, we study the contribution of the cohomology of the Lubin-Tate tower, via the wedge morphism, to the cohomology of the Rapoport-Zink tower. The $p$-adic Hodge theoretic description of the Rapoport-Zink tower, given by Scholze-Weinstein, is so far the best available technology for understanding the generic fiber of these spaces, as the moduli interpretation of these formal scheme is lost once we go to the generic fiber. It was therefore important to have a good understanding of the operation of the wedge morphism, which is defined in terms of the moduli property of these towers, on the generic fiber. Another motivation for such a Cartesian diagram comes from the relation to period morphisms. Let us explain this in more details. Let $G$ be a $p$-divisible group over $k$ of height $h$ and dimension $d$. The Grothendieck-Messing period morphism is a morphism from the adic generic fiber $\CM_{G,\eta}^{\infty,\ad}$ to the flag variety of rank $d$ quotients of the Dieudonn\'e module of $G$ (a rank $h$ free $W(k)$-module). In \cite{H6} we have studied various embeddings of flag varieties and in particular, we have shown than using exterior powers, one obtains a closed embedding of $\BP^{h-1}$ to the Grassmannian of rank $d$ quotients of a fixed rank $h$ vector bundle, denoted by $\BG r(h,d)$. Our main theorem implies in particular that the corresponding diagram of period morphism using exterior powers of $p$-divisible groups and the closed embedding  $\BP^{h-1}\into \BG r(h,d)$ is commutative. One should however note that this diagram is not Cartesian. In another work of ours, we will investigate the necessary modifications for making that diagram Cartesian (e.g., by incorporating the group actions of various reductive groups over these adic spaces in the diagram).\\ 

As a byproduct of our work on multilinear morphisms of vector bundles over the Fargues-Fontaine curve and the intermediary steps of the proof of the main theorem, we prove that there is a canonical isomorphism 
\[\ep^{r}\CE_{H}\cong\CE_{\ep^{r}H}\] where for a $p$-divisible group $G$ over a semiperfect ring $R$, we denote by $\CE_{G}$ the associated vector bundle over $\Proj\big(\bigoplus_{d\geq 0}\big(B^{+}_{\rm cris}(R)\big)^{\phi=p^{d}}\big)$. Here, as usual, $B^{+}_{\rm cris}(R)$ is one of Fontaine's period rings appearing in $p$-adic Hodge theory. The proof of this result, which can be stated without using Rapoport-Zink spaces and their $p$-adic Hodge theoretic description, requires the machinery of \cite{MR3272049} and our proof of the main theorem.\\

\textbf{Acknowledgments.} We would like to thank Pascal Boyer, Arthur-C\'esar Le Bras, Peter Scholze and Jared Weinstein for helpful conversations. 

\section{Preliminaries}

\subsection{Rank and the exterior power of a matrix}

In this subsection, we show that the rank of a matrix is determined by the rank of the exterior powers of it (see Lemma \ref{RankGivenByExtPow}). \\

\begin{notation}
\label{NotExtPowMat}
Let $R$ be a ring and $A$ an element of $\BM_{n}(R)$. Fix $1\leq d\leq n$. We denote by $\ep^{d}A\in\BM_{\binom{n}{d}}(R)$, the matrix whose entries are the $d$-minors of $A$.\xqed{\star}
\end{notation}

\begin{dfn}
Let $R$ be a ring and $A\in\BM_{n}(R)$. For $i\geq0$, the $i$-th \emph{determinantal ideal} of $A$, denoted by $\CU_{i}(A)$ is the ideal of $R$ generated by $i\times i$-minors of $A$. So, we have a chain:
\[0=\CU_{n+1}(A)\subset\CU_{n}(A)=\det(A)R\subset\CU_{n-1}(A)\subset\dots\]\[\dots\subset\CU_{1}(A)\subset\CU_{0}(A)=R\]\xqed{\blacktriangle}
\end{dfn}

\begin{lem}
\label{LemDetIdBaseChange}
Let $\phi:R\to S$ be a ring homomorphism, then for any $A\in\BM_{n}(R)$, and any $i\geq 0$, we have
\begin{align}
\label{EqDetIdBaseChange}
\CU_{i}\big(\phi(A)\big)=\CU_{i}(A)S
\end{align}
\end{lem}

\begin{proof}
This follows immediately from the definition.
\end{proof}



\begin{dfn}
\label{DefRankMatrix}
Let $R$ be a ring and $A$ an element of $\BM_{n}(R)$. We say that the \emph{rank} of $A$ is $r$ if all minors of size $r+1$ are zero and all minors of size $r$ generate the unit ideal of $R$. In other words, we have $\CU_{r}(A)=R$ and $\CU_{r+1}(A)=0$.\xqed{\blacktriangle}
\end{dfn}

\begin{rem}
\label{RemRankBaseChange}
\begin{enumerate}
\item Note that the rank of a matrix is not always defined.
\item It follows from Lemma \ref{LemDetIdBaseChange} that the rank of a matrix (when it is defined) is preserved under base change.\xqed{\lozenge}
\end{enumerate}
\end{rem}



\begin{lem}
\label{LemRankLocal}
Let $R$ be a ring and $A\in\BM_{n}(R)$. Then, $\rank(A)=r$ if and only if for all $\Fp\in\Spec(R)$, $\rank(A_{\Fp})=r$, where $A_{\Fp}$ is the image of $A$ in $\BM_{n}(R_{\Fp})$.
\end{lem}

\begin{proof}
If $\rank(A)=r$, then by previous remark, $A$ has rank $r$ in all localizations. Now assume that $A$ has rank $r$ in all localizations. This means that $\CU_{r+1}(A)$ is zero in all localizations (using \ref{EqDetIdBaseChange}), and so it is zero. Similarly, $\CU_{r}(A)$ is the unit ideal in all localizations, and so it is the unit ideal.
\end{proof}

\begin{lem}
\label{LemRankMatQuot}
Let $R$ be a ring and $A\in\BM_{n}(R)$. Then, $\rank(A)=r$ if and only if the cokernel of the $R$-linear morphism $:R^{n}\to R^{n}$ defined by $A$ is a projective $R$-module of rank $n-r$.
\end{lem}

\begin{proof}
Let us denote this cokernel by $W$. Recall that the $i$-th \emph{Fitting ideal}, $\text{Fit}_{i}(W)$, of $W$ is the ideal $\CU_{n-i}(A)$. \cite[\href{https://stacks.math.columbia.edu/tag/07ZD}{Lemma 07ZD}]{stacks-project} states that $W$ is finitely generated projective of rank $n-r$ if and only if $\text{Fit}_{n-r-1}(W)=0$ and $\text{Fit}_{n-r}(W)=R$, which is equivalent to $\rank(A)=r$
\end{proof}

\begin{lem}
\label{LemExSeqExtPowMod}
Let $R$ be a ring and let \[M\arrover{\alpha} N\arrover{\pi} W\to0\] be an exact sequence of finitely generated projective $R$-modules with $W$ of rank 1 and $M$ and $N$ of the same rank. Let $K$ be the kernel of $\pi$. Choose $d<\rank M$. Then we have a canonical exact sequence
\[\ep^{d}M\arrover{\ep^{d}\alpha}\ep^{d}N\to \ep^{d-1}K\otimes_{R}W\to 0.\]
\end{lem}

\begin{proof}
Let $K$ denote the kernel of $\pi$, then, it is finitely generated projective and we have a split short exact sequence
\[0\to K\to N\to W\to 0\] and so, since $W$ has rank $1$, we have $\ep^{d}N\cong \ep^{d}K\oplus \ep^{d-1}K\otimes W$. This means that the sequence
\[0\to\ep^{d}K\to\ep^{d}N\to \ep^{d-1}K\otimes W\to 0\] is exact. Now, since $\ep^{d}M\to \ep^{d}K$ is an epimorphism, it follows that the sequence 
\[\ep^{d}M\arrover{\ep^{d}\alpha}\ep^{d}N\to \ep^{d-1}K\otimes_{R}W\to 0\] is exact as desired. 
\end{proof}

\begin{lem}
\label{RankGivenByExtPow}
Let $R$ be a ring, $A\in\BM_{n}(R)$ and $d<n$. Then $A$ has rank $n-1$ if and only if the matrix $\ep^{d}A\in\BM_{\binom{n}{d}}(R)$ has rank $\binom{n-1}{d}$.
\end{lem}

\begin{proof}
Assume that $\rank(A)=n-1$. Then, by Lemma \ref{LemRankMatQuot}, we have an exact sequence 
\[R^{n}\arrover{A}R^{n}\to W\to 0\] where $W$ is free of rank $1$. By Lemma \ref{LemExSeqExtPowMod}, we have an exact sequence 
\[\ep^{d}(R^{n})\arrover{\ep^{d}A}\ep^{d}(R^{n})\to \ep^{d-1}K\otimes_{R}W\to 0\]
 where $K$ denotes the image of $A:R^{n}\to R^{n}$ and has rank $n-1$. The free $R$-module $\ep^{d-1}K\otimes_{R}W$ has rank $\binom{n-1}{d-1}$ and so, again by Lemma \ref{LemRankMatQuot}, $\ep^{d}A$ has rank $\binom{n}{d}-\binom{n-1}{d-1}=\binom{n-1}{d}$.\\
 
 Conversely, assume that $\ep^{d}A$ has rank $\binom{n-1}{d}$. By Lemma \ref{LemRankLocal}, we can assume that $R$ is a local ring with maximal ideal $\Fm$ and residue field $\kappa$. Let $r$ be the rank of $A$ modulo $\Fm$ (that we denote by $\bar{A}$). Thus $\bar{A}$ is equivalent (over $\kappa$) to the matrix
 \[\begin{pmatrix}I_{r}&\mathbf{0}\\\mathbf{0}&\mathbf{0}
\end{pmatrix}\]
(here $I_{r}$ is the identity matrix of size $r$) which implies that $\ep^{d}A$ modulo $\Fm$ is equivalent to the matrix
 \[\begin{pmatrix}I_{\binom{r}{d}}&\mathbf{0}\\\mathbf{0}&\mathbf{0}
\end{pmatrix}\]
We therefore have:
\[\binom{n-1}{d}=\rank(\ep^{d}A)=\rank(\ep^{d}\bar{A})=\binom{r}{d}\] where the second equality follows from Remark \ref{RemRankBaseChange} (ii). Since $d<n$, this implies that $r=n-1$.\\

So, $\CU_{n-1}(\bar{A})=\kappa$, and therefore the determinantal ideal $\CU_{n-1}(A)$ is the unit ideal. In order to prove that $A$ has rank $n-1$, we have to show that $\det(A)=0$. Since there is an invertible minor of size $n-1$, by a generalized pivot method (see e.g. \cite[5.9]{MR3408454}), we can show that $A$ is equivalent to the matrix
 \[\begin{pmatrix}I_{n-1}&\mathbf{0}\\\mathbf{0}&\det(A)
\end{pmatrix}\]
and so, $\ep^{d}A$ is equivalent to the matrix 

 \[\begin{pmatrix}I_{\binom{n-1}{d}}&\mathbf{0}\\\mathbf{0}&\det(A)I_{\binom{n-1}{d-1}}
\end{pmatrix}\]
Since $\ep^{d}A$ has rank $\binom{n-1}{d}$, this implies that $\det(A)=0$.
\end{proof}

\subsection{Elements from $p$-adic Hodge theory}

In this subsection we recall some definitions and results from $p$-adic Hodge theory.

\begin{dfn}
Let $R$ be a ring of characteristic $p>0$. It is called \emph{semiperfect} if its Frobenius is surjective.\xqed{\blacktriangle}
\end{dfn}

We have the following result of Fontaine:

\begin{prop}[\cite{MR3272049}, Proposition 4.1.3.] Let $R$ be a semiperfect ring of characteristic $p>0$. Let $R^{\flat}:=\invlim_{\rm Frob}R$ be the tilt of $R$. Denote by $A_{\rm cris}(R)$ the $p$-adic completion of the PD hull of the surjection $W(R^{\flat})\onto R$. Then, $A_{\rm cris}(R)$ is the universal $p$-adically complete PD thickening of $R$, and its construction is functorial in $R$. In particular it has a natural Frobenius morphism $\phi:A_{\rm cris}(R)\to A_{\rm cris}(R)$ coming from the Frobenius of $R$.
\end{prop}
	
Set $B^{+}_{\rm cris}(A):=A_{\rm cris}(R)[1/p]$.

\begin{dfn}
\label{DieuMod}
Let $R$ be a semiperfect ring of characteristic $p>0$.
\begin{enumerate}
\item A \emph{Dieudonn\'e module} over $R$ is a finitely generated projective $A_{\rm cris}(R)$-module $M$ together with $A_{\rm cris}(R)$-linear homomorphisms
\begin{align*}
F&:M\otimes_{A_{\rm cris}(R),\phi}A_{\rm cris}(R)\to M\\
V&:M\to M\otimes_{A_{\rm cris}(R),\phi}A_{\rm cris}(R)
\end{align*}
such that $FV=p=VF$. 
\item A \emph{rational Dieudonn\'e module} over $R$ is a finitely generated projective $B^{+}_{\rm cris}(R)$-module $M$ with morphisms $F$ and $V$ as above (replacing $A$ with $B^{+}$ everywhere).
\item An \emph{isocrystal} over $R$ is a finitely generated projective $B^{+}_{\rm cris}(R)$-module $N$ together with a $B^{+}_{\rm cris}(R)$-linear isomorphism \[F:N\otimes_{B^{+}_{\rm cris}(R),\phi}B^{+}_{\rm cris}(R)\to N\]We say that $N$ is \emph{integral} if there is a finitely generated projective $A_{\rm cris}(R)$-module $M$ such that $N=M[1/p]$ and $F(M\otimes_{A_{\rm cris}(R),\phi}A_{\rm cris}(R))\subset M$.
\end{enumerate}
By abuse of notation, we will denote by $F$ the  $\phi$-semilinear morphism sending $m$ to $F(m\otimes 1)$. We will also denote by $M^{\sigma}$ the base change $M\otimes_{A_{\rm cris}(R),\phi}A_{\rm cris}(R)$ or $M\otimes_{B^{+}_{\rm cris}(R),\phi}B^{+}_{\rm cris}(R)$.\xqed{\blacktriangle}
\end{dfn}

\begin{rem}
Note that since $A_{\rm cris}(R)$ is $p$-torsion-free,  $F:M\to M$ is injective. In the case of rational Dieudonn\'e modules, $F$ and $V$ (the linear versions) are even isomorphisms. Therefore, a rational Dieudonn\'e module is an isocrystal, by forgetting $V$.\xqed{\lozenge}
\end{rem}

\begin{ex}
Let $k$ be a perfect field of characteristic $p$ and $G$ a $p$-divisible group over $k$. Let $R$ be a semiperfect $k$-algebra, then the finite free $A_{\rm cris}(R)$-module $\BD(G)\otimes_{W(k)}A_{\rm cris}(R)$ has a natural Frobenius and Verschiebung (extending $F$ and $V$ of $\BD(G)$), making it a Dieudonn\'e module over $R$. The $\phi$-semilinear Frobenius 
\[F:\BD(G)\otimes_{W(k)}A_{\rm cris}(R)\to \BD(G)\otimes_{W(k)}A_{\rm cris}(R)\] is given by the formula $x\otimes a\mapsto F(x)\otimes \phi(a)$. \xqed{\blacksquare}
\end{ex}

\begin{dfn}
Let $R$ be a semiperfect ring of characteristic $p>0$ and let $G$ be a $p$-divisible group over $R$. The Dieudonn\'e module of $G$, denoted by $\BD(G)$ is the evaluation of the crystal of $G$ on the PD thickening $A_{\rm cris}(R)\to R$. This defines a functor from the category of $p$-divisible groups over $R$ to the category of Dieudonn\'e modules over $R$.\xqed{\blacktriangle}
\end{dfn}

\begin{dfn}
\label{DefGradedIsoCry}
Let $R$ be a semiperfect ring of characterisitc $p>0$. We denote by $\CP_{R}$ the graded $\BQ_{p}$-algebra \[\CP_{R}=\bigoplus_{d\geq 0}\big(B^{+}_{\rm cris}(R)\big)^{\phi=p^{d}}\]
Let $(N,F)$ be an isocrystal over $R$. We define the graded $\CP_{R}$-module
\[N_{\rm gr}:=\bigoplus_{d\geq 0}N^{F=p^{d+1}}\] and denote by $\CE_{N,F}$ (or $\CE_{N}$) the associated vector bundle over $\Proj\CP_{R}$. Note that the degree $d$ elements of $N_{\rm gr}$ are $N^{F=p^{d+1}}$.\xqed{\blacktriangle}
\end{dfn}

Let us fix a complete and algebraically closed extension $\mathbf{C}$ of $\BQ_{p}$. Recall (cf. \cite{courbe} Ch. 6, \S 6.1) that the Fargues-Fontaine curve is $X:=\Proj \CP_{\CO_{\mathbf{C}}/p}$. Throughout this paper, we will reserve the letter $X$ for this curve.\\

There is a natural morphism $\Theta:B^{+}_{\rm cris}(\CO_{\mathbf{C}}/p)\to \mathbf{C}$ (Fontaine's morphism), which defines a closed embedding $i_{\infty}:\{\infty\}\to X$. In fact, if $k$ is a perfect field of characteristic $p$ and $(R,R^{+})$ is a perfectoid affinoid $(W(k)[1/p],W(k))$-algebra, we have Fontaine's morphism \[\Theta:B^{+}_{\rm cris}(R^{+}/p)\to R\]

Let $(N,F)$ be an isocrystal over $\CO_{\mathbf{C}}/p$. We have a canonical isomorphism 
\begin{align}
\label{EqGlobSecVecBun}
\Gamma(X,\CE_{N})\cong N^{F=p}
\end{align}

\subsection{Universal cover of $p$-divisible groups}

\begin{dfn}
\label{DefTateModSh}
Let $S$ be a scheme and $G$ a $p$-divisible group over $S$. The $p$-adic Tate module of $G$ is the sheaf of $\BZ_{p}$-modules $T_p(G):=\invlim_{n}G[p^{n}]$.\xqed{\blacktriangle}
\end{dfn}

\begin{rem}
\label{RemTateModAltDef}
Note that we have a canonical isomorphism of $\BZ_{p}$-sheaves:
\begin{align}
\label{IsomTateModAltDef}
T_p(G)\cong\innHom_{S}(\BQ_{p}/\BZ_{p},G)
\end{align}\xqed{\lozenge}
\end{rem}

\begin{notation}
Let $R$ be a ring on which $p$ is topologically nilpotent, and denote by $\text{Nilp}_{R}^{\rm{op}}$ the category opposite of the category of $R$-algebras on which $p$ is nilpotent.\xqed{\star}
\end{notation}

\begin{dfn}
\label{DefAdGenFibFun}
Let $(R,R^{+})$ be a complete affinoid $\big(W(k)[1/p],W(k)\big)$-algebra. Define the \emph{adic generic fiber functor}
\[(\_)_{\eta}^{\rm ad}:(\text{Nilp}_{R^{+}}^{\rm op})^{\sim}\to (\text{CAff}_{(R,R^{+})}^{\rm op})^{\sim}\] by sending a sheaf $\CF$ to the sheafification of 
\[(S,S^{+})\mapsto \dirlim_{S_{0}\subset S^{+}}\invlim_{n}\CF(S_{0}/p^{n})\]where $\text{CAff}_{(R,R^{+})}$ is the category of complete affinoid $(R,R^{+})$-algebras and the direct limit runs over all open and bounded sub-$R^{+}$-algebras $S_{0}$ of $S^{+}$ (for a discussion on the topology of $(\text{Nilp}_{R^{+}}^{\rm op})^{\sim}$ see \cite[Ch. 2]{MR1393439} and for more details on the adic generic functor, see \cite[\S2.2]{MR3272049}).\xqed{\blacktriangle}
\end{dfn}

\begin{dfn}
Let $R$ be a ring on which $p$ is topologically nilpotent and let $G$ be a $p$-divisible group over $R$. We will consider $G$ as the sheaf on $\text{Nilp}_{R}^{\rm{op}}$, sending an $R$-algebra $S$ to $\dirlim G[p^{n}](S)$. The \emph{universal cover} of $G$, denoted by $\tilde{G}$ is the sheaf of $\BQ_{p}$-vector spaces on $\text{Nilp}_{R}^{\rm{op}}$ that sends $S$ to \[\invlim_{p\cdot}G(S)\] where the transition morphisms are the multiplication-by-$p$ morphism (cf. \cite{courbe}, Ch. 4, D\'efinition 4.5.1., or \cite{MR3272049}, \S3). We extend this functor to $R$-algebras on which $p$ is topologically nilpotent, by sending such an $R$-algebra $S$ to the limit
\[\invlim_{n}\tilde{G}(S/p^{n})\]\xqed{\blacktriangle}
\end{dfn}

Let us list some properties of the universal cover that we will use throughout the paper:

\begin{prop}
\label{PropPropUnivBTGrp}
Let $R$ be a ring on which $p$ is topologically nilpotent and fix a $p$-divisible group $G$ over $R$. We denote by $\BD(G)$ the Dieudonn\'e module of $G$ over $R$ (when $R$ is semiperfect).
\begin{enumerate}[(1)]
\item if $p$ is nilpotent in $R$, then we have a canonical isomorphism 
\begin{align}
\tilde{G}(R)\cong\Hom_{R}(\BQ_{p}/\BZ_{p},G)[1/p]
\end{align}
\item if $R$ is $p$-adically complete, then we have a canonical isomorphism 
\begin{align}
\tilde{G}(R)\cong\tilde{G}(R/p)\cong \Hom_{R/p}(\BQ_{p}/\BZ_{p},G)[1/p]
\end{align}
\item
there is a canonical isomorphism of $\BQ_{p}$-sheaves \[\tilde{G}\cong T_{p}(G)[1/p]\]where again $T_p(G)$ is the $p$-adic Tate module of $G$, seen as a sheaf.
\item
assume that $R$ is an $f$-semiperfect ring (meaning that the tilt $R^{\flat}$ is $f$-adic), then we have a canonical isomorphism
\begin{align}
\label{GlobSecUnivCov}
\tilde{G}(R)\cong \big(\BD(G)[1/p]\big)^{F=p}
\end{align}
\item
let $R$ be a perfect field. Then, the universal cover $\tilde{G}$ lifts uniquely to $W(R)$ and for any perfectoid affinoid $(W(R)[1/p],W(R))$-algebra $(S,S^{+})$, we have a canonical isomorphism
\begin{align}
\label{AdGenFibGlobSec}
\tilde{G}_{\eta}^{\rm ad}(S,S^{+})\cong \big(\BD(G)\otimes_{W(k)}B^{+}_{\rm cris}(S^{+}/p)\big)^{F=p}
\end{align}
\end{enumerate}
\end{prop}

\begin{proof}
See \cite[\S3 and \S5]{MR3272049}.
\end{proof}

\begin{rem}
Note that when $p$ is merely topologically nilpotent, we extend the definitions of $\tilde{G}$ and $T_{p}(G)$ by taking inverse limits over truncations by powers of $p$. So, in the isomorphism of part (3) of the Proposition, the localization at $1/p$ is \emph{before} taking inverse limits. In other words, for a ring $R$ in which $p$ is topologically nilpotent (and not nilpotent), we have 
\begin{align}
\label{EqIsomTildGTate}
\tilde{G}(R)\cong \invlim_{n}\tilde{G}(R/p^{n})\cong \invlim_{n}\big(T_{p}(G)(R/p^{n})[1/p]\big)\cong \big(T_{p}G[1/p]\big)(R)
\end{align}
whereas, 
\[\big(T_{p}(G)(R)\big)[1/p]\cong \big(\invlim_{n}T_{p}(G)(R/p^{n})\big)[1/p]\]
and these two inverse limits are not isomorphic.
Isomorphism (\ref{EqIsomTildGTate}) yields a canonical embedding of $\BZ_{p}$-sheaves:
\begin{align}
T_p(G)\into\tilde{G}
\end{align}\xqed{\lozenge}
\end{rem}

\begin{rem}
As we said in part (5) of Proposition \ref{PropPropUnivBTGrp}, when $R$ is a perfect field, the universal cover $\tilde{G}$ of $G$ lifts uniquely to $W(R)$, and we will denote the lift by $\tilde{G}$ as well.\xqed{\lozenge}
\end{rem}

\begin{lem}
\label{LemGlobSecVectUnivCov}
Let $k$ be a perfect field of characteristic $p>0$ and $G$ a $p$-divisible group over $k$. Let $\tilde{G}$ be the unique lift of the universal cover of $G$ to $W(k)$, and $\CE_{G}$ the vector bundle over the Fargues-Fontaine curve $X$ associated with $G$, i.e., $\CE_{\BD(G)\otimes_{W(k)}B^{+}_{\rm cris}(\CO_{\mathbf{C}}/p)}$. Then, for any natural number $n$, we have a canonical bijection:
\begin{align}
\label{EqGlobSecVectBunUnivCov}
\tilde{G}^{\rm ad}_{\eta}(\mathbf{C},\CO_{\mathbf{C}})^{\times n}\cong \Hom_{X}(\CO_{X}^{\oplus n},\CE_{G})
\end{align}
\end{lem}

\begin{proof}
This following by combining isomorphisms (\ref{EqGlobSecVecBun}) and (\ref{AdGenFibGlobSec}), and observing that we have \[\Hom_{X}(\CO_{X},\CE_{G})\cong \Gamma(X,\CE_{G}).\]
\end{proof}

\begin{lem}
\label{LemCommDiagQLogTheta}
Let $G$ be a $p$-divisible group over a perfect field $k$ of characteristic $p$, and $(R,R^{+})$ a perfectoid affinoid $\big(W(k)[1/p],W(k)\big)$-algebra. The following diagram is commutative

\[\xymatrix{\tilde{G}_{\eta}^{\rm ad}(R,R^{+})\ar[d]_{\rm qlog}\ar[r]^{\cong\qquad\quad}&\big(\BD(G)\otimes_{W(k)}B^{+}_{\rm cris}(R^{+}/p)\big)^{F=p}\ar[dl]^{\Theta}\\
\BD(G)\otimes_{W(k)}R&}\]where $\qlog:\tilde{G}_{\eta}^{\ad}\to \BD(H)\otimes\BG_{a}$ is the quasi-logarithm (see \cite[\S3.2 and \S6.3]{MR3272049}).
\end{lem}

\begin{proof}
This follows from Lemma 3.5.1 and Theorem 4.1.4. in \cite{MR3272049}.
\end{proof}

\section{Multilinear Theory}

\subsection{Multilinear morphisms of graded modules}

In this subsection, we define multilinear morphisms of graded modules over graded rings and show that they induce multilinear morphisms between their associated quasi-coherent sheaves over the $\Proj$.\\

\begin{notation}
\label{NotationLambdaMap}
Let $\Sigma$ and $\Delta$ be sets and $\varrho:\Sigma^{\times r}\to \Delta$ a map. Let $h\geq r$. We denote by $\Lambda_{r}\varrho$ (or even $\Lambda_{r}$ if $\varrho$ is understood from the context), the following map:
\begin{align*}
\Lambda_{r}\varrho:\Sigma^{\times h}&\to \Delta^{\times \binom{h}{r}}\\
(x_{1},\dots,x_{h})&\mapsto \big(\varrho(x_{i_{1}},\dots,x_{i_{r}})\big)_{1\leq i_{1}<\dots<i_{r}\leq h}
\end{align*}
Now, if $\mathscr{C}$ is a category, $\Sigma,\Delta:\mathscr{C}\to \Ens$ are functors and $\varrho:\Sigma^{\times r}\to \Delta$ is a natural trasnformation, we can define the natural transformation $\Lambda_{r}\varrho$ in the same fashion.\xqed{\star}
\end{notation}

\begin{dfn}
\label{DefGradMultMor}
Let $S=\bigoplus_{d\geq 0}S_{d}$ be a graded ring and $M, M_{1},\dots,M_{r},N$ graded $S$-modules. A \emph{graded multilinear morphism} $$\tau:M_{1}\times\dots\times M_{r}\to N$$ is a multilinear morphism of  $S_{0}$-modules such that for all $d_{1},\dots,d_{r}\geq 0$, we have \[\tau(M_{d_{1}}\times\dots\times M_{d_{r}})\subset N_{d_{1}+\dots+d_{r}}\]We denote by $\Mult_{\rm gr}(M_{1}\times\dots\times M_{r},N)$ the abelian group of all such multilinear morphisms. We denote by $\Alt_{\rm gr}(M^{\times r},N)$ (and respectively $\Sym_{\rm gr}(M^{\times r},N)$) the subgroup of $\Mult_{\rm gr}(M^{\times r},N)$ consisting of alternating (respectively symmetric) elements. When $r=1$, we obtain the usual notion of graded morphism of graded modules.\xqed{\blacktriangle}
\end{dfn}

\begin{dfn}
Let $S=\bigoplus_{d\geq 0}S_{d}$ be a graded ring and $M, N$ graded $S$-modules. For $i\geq 0$, we denote by $i$-$\Hom(M,N)$ the group $\Hom_{\rm gr}(M,N[i])$. We call elements of this group \emph{$i$-grarded morphisms}. Using notations in \cite{MR0379473}, we denote the graded $S$-module $\oplus_{i\geq 0}\, i$-$\Hom(M,N)$ by $*\Hom_{S}(M,N)$. \xqed{\blacktriangle}
\end{dfn}

\begin{lem}
\label{LemAdjGradMod}
Let $S=\bigoplus_{d\geq0}S_{d}$ be a graded ring and $M_{1},\dots,M_{r},N$ graded $S$-modules. Under the canonical isomorphism 
\begin{align}
\label{EqGenAdj}
\Theta:\Mult_{S}(M_{1}\times\dots\times M_{r},N)\cong\Mult_{S}(M_{1}\times\dots\times M_{r-1},\Hom(M_{r},N))
\end{align}
The subgroup $\Mult_{\rm gr}(M_{1}\times\dots\times M_{r},N)$ is mapped onto the subgroup $\Mult_{\rm gr}(M_{1} \times \dots \times M_{r-1},*\Hom_{S}(M_{r},N))$.
\end{lem}

\begin{proof}
Recall that the isomorphism (\ref{EqGenAdj}) sends $\phi$ to 
\[\Theta(\phi):(m_{1},\dots,m_{r-1})\mapsto [m_{r}\mapsto \phi(m_{1},\dots,m_{r})]\]
Take $\phi\in \Mult_{S}(M_{1}\times\dots\times M_{r},N)$, then $\phi$ is graded if and only if $$\phi(M_{d_{1}},\dots,M_{d_{r}})\subset N_{d_{1}+\dots+d_{r}}$$ if and only if $$\Theta(\phi)(M_{d_{1}},\dots,M_{d_{r-1}})(M_{d_{r}})\subset N_{d_{1}+\dots+d_{r}}$$ if and only if $$\Theta(\phi)(M_{d_{1}},\dots,M_{d_{r-1}})\subset (d_{1}+\dots+d_{r-1})\text{-}\Hom(M_{r},N)$$ if and only if $\Theta(\phi)$ is graded.
\end{proof}

\begin{lem}
\label{LemTildeGraded}
Let $S=\bigoplus_{d\geq0}S_{d}$ be a graded ring and $M, N$ graded $S$-modules. Set $Y:=\Proj S$. For a graded $S$-module $P$, we denote by $\tilde{P}$ the $\CO_{Y}$-module associated with $P$. There is a canonical and functorial morphism of $\CO_{Y}$-modules
\begin{align}
\label{EqTildeGraded}
\big(*\Hom_{S}(M,N)\big)^{\widetilde{}}\to \innHom_{\CO_{Y}}(\tilde{M},\tilde{N})
\end{align}
\end{lem}

\begin{proof}
Let $f\in S_{d}$ be a homogenous element of degree $d$ and $U=D_{+}(f)$ the special open subset of $Y$ attached to $f$. For a graded $S$-module $P$ we have $\tilde{P}_{|_{U}}\cong \widetilde{P_{(f)}}$ and so, it is enough to show that we have a canonical and functorial morphism 
\[\big(*\Hom_{S}(M,N)\big)_{(f)}\to \Hom_{S_{(f)}}(M_{(f)},N_{(f)})\]
compatible with restrictions $D_{+}(f)\to D_{+}(g)$. So, let $\frac{\phi}{f^{n}}$ be a degree zero quotient in $\big(*\Hom_{S}(M,N)\big)_{(f)}$, so, $\phi$ has degree $nd$. We send this quotient to the morphism $M_{(f)}\to N_{(f)}$ that sends $\frac{m_{id}}{f^{i}}$ to $\frac{\phi(m_{id})}{f^{n+i}}$. It is now straightforward to check that all the required conditions are satisfied and thus, we have the desired morphism 
\[\big(*\Hom_{S}(M,N)\big)^{\widetilde{}}\to \innHom_{\CO_{Y}}(\tilde{M},\tilde{N})\]
\end{proof}

\begin{prop}
\label{PropGrShGr}
Let $S=\bigoplus_{d\geq0}S_{d}$ be a graded ring and $M_{1},\dots,M_{r}, M, N$ graded $S$-modules. Set $Y=\Proj S$ and for $i=1,\dots,r$, let $\tilde{M}_{i}$ (respectively $\tilde{M}, \tilde{N}$) be the $\CO_{Y}$-module associated with $M_{i}$ (respectively $M, N$). Then there are canonical and functorial homomorphisms
\begin{align}
\label{EqGrShGr}
\Mult_{\rm gr}(M_{1}\times\dots\times M_{r},N)&\to \Mult_{\CO_{Y}}(\tilde{M}_{1}\times\dots\times\tilde{M}_{r},\tilde{N})\\
\Alt_{\rm gr}(M^{\times r},N)&\to \Alt_{\CO_{Y}}(\tilde{M}^{\times r},\tilde{N})\\
\Sym_{\rm gr}(M^{\times r},N)&\to \Sym_{\CO_{Y}}(\tilde{M}^{\times r},\tilde{N})
\end{align}
where on the right hand sides we have respectively the group of multilinear, alternating and symmetric morphisms of $\CO_{Y}$-modules.
\end{prop}

\begin{proof}
We will prove the first statement and the other two will be similar (and in fact follow from it). We are going to prove the statement by induction on $r$. For $r=1$, this follows from functoriality of the $(\_)^{\tilde{}}$ construction. So, assume the result holds for $r$ and we want to prove it for $r+1$. By Lemma \ref{LemAdjGradMod}, we have an isomorphism
\[\Mult_{\rm gr}(M_{1}\times\dots\times M_{r+1},N)\cong \Mult_{\rm gr}(M_{1}\times\dots\times M_{r},*\Hom_{S}(M_{r+1},N))\] By induction hypothesis, there is a morphism from the right hand side to 
\[\Mult_{\CO_{Y}}\big(\tilde{M}_{1}\times\dots\times\tilde{M}_{r},\big(*\Hom_{S}(M_{r+1},N)\big)^{\widetilde{}}\big)\] and composing with morphism (\ref{EqTildeGraded}), we obtain a morphism to 
\[\Mult_{\CO_{Y}}\big(\tilde{M}_{1}\times\dots\times\tilde{M}_{r},\innHom_{\CO_{Y}}(\tilde{M},\tilde{N})\big)\] which is isomorphic to 
\[\Mult_{\CO_{Y}}(\tilde{M}_{1}\times\dots\times\tilde{M}_{r+1},\tilde{N})\]All these morphisms being canonical and functorial, we obtain the desired morphism
\[\Mult_{\rm gr}(M_{1}\times\dots\times M_{r+1},N)\to \Mult_{\CO_{Y}}(\tilde{M}_{1}\times\dots\times\tilde{M}_{r+1},\tilde{N})\] and the proof is achieved.
\end{proof}

\subsection{Multilinear morphisms of vector bundles on the Fargues-Fontaine curve}

In this subsection, we show how multilinear morphisms of Dieudonn\'e modules and isocrystals define, in a natural way, multilinear morphisms between their associated vector bundles on the Fargues-Fontaine curve of $p$-adic Hodge theory.\\

Throughout this subsection, $R$ is a semiperfect ring of characterisitc $p>0$.\\

\begin{dfn}
\label{MultLinDieuMod}
Let $M, M_{1},\dots,M_{r}$ be rational Dieudonn\'e modules over $R$ and $N$ an isocrystal over $R$. A \emph{multilinear morphism} $$\tau:M_{1}\times\dots\times M_{r}\to N$$ is a $B^{+}_{\rm cris}(R)$-multilinear morphism of $B^{+}_{\rm cris}(R)$-modules making the following diagrams commutative ($i=1,\dots,r$):
\[\xymatrix{M_{1}^{\sigma}\times\dots\times M_{r}^{\sigma}\ar[rr]^{\tau^{\sigma}}&&N^{\sigma}\ar[dd]^{F}\\
M_{1}\times\dots\times M_{i-1}\times M_{i}^{\sigma}\times M_{i+1}\times\dots\times M_{r}\ar[d]_{\Id\times\dots\times\Id\times F\times\Id\times\dots\times\Id}\ar[u]^{V\times\dots\times V\times\Id\times V\times\dots\times V}&&\\
M_{1}\times\dots\times M_{r}\ar[rr]_{\tau}&&N}\]
In other words, for all $i=1,\dots,r$ and all $x_{1},\dots,x_{r}$ (in the appropriate module!), we have 
\begin{align}
\label{EqMultLinDieuMod}
\nonumber F\big(\tau^{\sigma}(Vx_{1},\dots,Vx_{i-1},x_{i},Vx_{i+1},\dots, Vx_{r})\big)=\\\tau(x_{1},\dots,x_{i-1},Fx_{i},x_{i+1},\dots,x_{r})
\end{align}We will denote the $B^{+}_{\rm cris}(R)$-module of all such multilinear morphisms with $$\Mult_{R}(M_{1}\times\dots\times M_{r},N)$$ and if the chance of confusion is little, we drop $R$ from the notation. We denote by $\Alt_{R}(M^{\times r},N)$ (respectively $\Sym_{R}(M^{\times r},N)$) the subset of $\Mult_{R}(M^{\times r},N)$ consisting of alternating (respectively symmetric) elements.\xqed{\blacktriangle}
\end{dfn}

\begin{prop}
\label{PropMultGraded}
Let $M_{1},\dots,M_{r}$ be rational Dieudonn\'e modules over $R$ and $N$ an isocrystal over $R$. Then, every multilinear morphism $$M_{1}\times\dots\times M_{r}\to N$$ induces, by restriction, a graded $\BQ_{p}$-multilinear morphism (in the sense of Definition \ref{DefGradMultMor}, see also Definition \ref{DefGradedIsoCry})
\[M_{1,\rm gr}\times\dots\times M_{r,\rm gr}\to N_{\rm gr}\]in other words, restriction defines a homomorphism
\begin{align}
\label{EqInjDieuGrad}
\Mult(M_{1}\times\dots\times M_{r},N)\to \Mult_{\rm gr}(M_{1,\rm gr}\times\dots\times M_{r,\rm gr}\to N_{\rm gr})
\end{align}
\end{prop}

\begin{proof}
Let $\tau:M_{1}\times\dots\times M_{r}\to N$ be a multilinear morphism and take elements $m_{d_{i}}\in M_{i,d_{i}}$ ($i=1,\dots,r$) , i.e., $F(m_{d_{i}}\otimes 1)=p^{d_{i}+1}m_{d_{i}}$. We have to show that 
$\tau(m_{d_{1}},\dots,m_{d_{r}})$ is of degree $p^{d_{1}+\dots+d_{r}}$, i.e., lies in $N^{F=p^{d_{1}+\dots+d_{r}+1}}$. Since we have $F(m_{d_{i}}\otimes 1)=p^{d_{i}+1}m_{d_{i}}=F(Vp^{d_{i}}m_{d_{i}})$, and $F$ is injective, we have for all $i$
\[m_{d_{i}}\otimes 1=V(p^{d_{i}}m_{d_{i}})\] and so we have the following series of equalities:
\[F\big(\tau(m_{d_{1}},\dots,m_{d_{r}})\otimes 1\big)=F\tau^{\sigma}(m_{d_{1}}\otimes 1,\dots,m_{d_{r}}\otimes 1)=\]\[F\tau^{\sigma}\big(m_{d_{1}}\otimes 1,V(p^{d_{i}}m_{d_{2}}),\dots,V(p^{d_{i}}m_{d_{r}})\big)\ovset{(\ref{EqMultLinDieuMod})}{=}\]
\[\tau\big(F(m_{d_{1}}\otimes 1),p^{d_{2}}m_{d_{2}},\dots,p^{d_{r}}m_{d_{r}}\big)=\tau\big(p^{d_{1}+1}m_{d_{1}},p^{d_{2}}m_{d_{2}},\dots,p^{d_{r}}m_{d_{r}}\big)=\]\[p^{d_{1}+\dots+d_{r}+1}\tau(m_{d_{1}},\dots,m_{d_{r}})\] and the proof is achieved.
\end{proof}

\begin{cor}
\label{CorMultPresF=p}
Let $M_{1},\dots,M_{r}$ be rational Dieudonn\'e modules over $R$ and $N$ an isocrystal over $R$. Let $\tau:M_{1}\times\dots\times M_{r}\to N$ be a multilinear morphism, then the restriction of $\tau$ to $M_{1}^{F=p}\times\dots\times M_{r}^{F=p}$ factors through the subset $N^{F=p}$ of $N$, i.e., we have the following commutative diagram:
\[\xymatrix{M_{1}\times\dots\times M_{r}\ar[rr]^{\qquad\tau}&&N\\
M_{1}^{F=p}\times\dots\times M_{1}^{F=p}\ar[rr]_{\qquad\tau}\ar@{^{(}->}[u]&&N_{1}^{F=p}\ar@{^{(}->}[u]}\]So, we have a homomorphism \begin{align}
\label{EqMultPresF=p}\Mult(M_{1}\times\dots\times M_{r},N)\to\Mult(M_{1}^{F=p}\times\dots\times M_{r}^{F=p},N^{F=p})\end{align}
\end{cor}

\begin{proof}
This follows from the definition of graded multilinear morphism and the proposition.
\end{proof}

\begin{rem}
We have statements similar to the previous proposition and corollary for alternating and symmetric multilinear morphisms. So, for example, if $M$ is a rational Dieudonn\'e module over $R$ and $N$ is an isocrystal over $R$, then we have a homomorphism:
\begin{align}
\label{NatHomMultDieuCrys}
\Alt(M^{r},N)\to\Alt\big((M^{F=p})^{\times r},N^{F=p}\big)
\end{align}
\xqed{\lozenge}
\end{rem}

\begin{cor}
\label{CorUnivAltMor}
Let $M$ be a rational Dieudonn\'e module over $R$ and fix a natural number $r$. Assume that $\ep^{r}_{B^{+}_{\rm cris}(R)}M$ has a Frobenius $F$ (i.e., is an isocrystal over $R$) such that the universal alternating morphism of $B^{+}_{\rm cris}(R)$-modules
\begin{align*}
\lambda_{M}:M^{\times r}&\to \ep^{r}_{B^{+}_{\rm cris}(R)}M\\
(m_{1},\dots,m_{r})&\mapsto m_{1}\wedge\dots\wedge m_{r}
\end{align*}
is a multilinear morphism in the sense of Definition \ref{MultLinDieuMod}. Then, this alternating morphism induces a morphism
\begin{align}\label{MapFrob=p}\ep^{r}_{\BQ_{p}}(M^{F=p})\to\big(\ep^{r}_{B^{+}_{\rm cris}(R)}M\big)^{F=p}\end{align}
\end{cor}

\begin{proof}
If in (\ref{NatHomMultDieuCrys}) we replace $N$ with $\ep^{r}M$, then the image of $\lambda_{M}$ is an alternating morphism 
\begin{align}
\label{EqCanAltMorF=p}
\lambda_{M}:(M^{F=p})^{\times r}\to (\ep^{r}M)^{F=p}\end{align} and therefore, it induces a canonical homomorphism \[\ep^{r}_{\BQ_{p}}(M^{F=p})\to\big(\ep^{r}M\big)^{F=p}\] as desired.
\end{proof}


\begin{cons}
Using Notations \ref{NotationLambdaMap}, from 
$$\lambda_{M}: (M^{F=p})^{\times r}\to (\ep^{r}M)^{F=p}$$ we obtain a morphism
\begin{align}
\label{MapExtPowerFrob=p}
\Lambda_{r}\lambda_{M}:(M^{F=p})^{\times h}\to \Big((\ep^{r}M)^{F=p}\Big)^{\times\binom{h}{r}}\end{align}\xqed{\blacktriangledown}
\end{cons}




\begin{prop}
\label{PropMultIndSh}
Let $M, M_{1},\dots,M_{r}$ be rational Dieudonn\'e modules over $R$ and $N$ an isocrystal over $R$. There are natural and functorial homomorphisms
\begin{align}\label{EqMultIndSh}
\Mult(M_{1}\times\dots\times M_{r},N)&\to\Mult(\CE_{M_{1}}\times\dots\times\CE_{M_{r}},\CE_{N})\\
\Alt(M^{\times r},N)&\to\Alt(\CE_{M}^{\times r},\CE_{N})\\
\Sym(M^{\times r},N)&\to\Sym(\CE_{M}^{\times r},\CE_{N})
\end{align}
\end{prop}

\begin{proof}
These homomorphisms are given by composing homomorphisms (\ref{EqGrShGr}) and (\ref{EqInjDieuGrad}).
\end{proof}

\begin{rem}
Let  $R$ be $\CO_{\mathbf{C}}/p$. Taking global sections, induces a homomorphism (using the canonical isomorphism (\ref{EqGlobSecVecBun}))
\[\Mult(\CE_{M_{1}}\times\dots\times\CE_{M_{r}},\CE_{N})\to\Mult(M_{1}^{F=p}\times\dots\times M_{r}^{F=p},N^{F=p})\]whose composition with homomorphism (\ref{EqMultIndSh}) is nothing but homomorphism (\ref{EqMultPresF=p}). \xqed{\lozenge}
\end{rem}

\begin{cor}
Let $M$ be a rational Dieudonn\'e module over $R$ satisfying the condition of Corollary \ref{CorUnivAltMor}. Then, we have a canonical morphism of $\CO_{X}$-modules:
\begin{align}
\label{MapExtVecBun}
\mathscr{L}_{M}:\ep^{r}_{\CO_{X}}\CE_{M}\to\CE_{\ep^{r}M}
\end{align}
\end{cor}

\begin{proof}
By Proposition \ref{PropMultIndSh} we have a homomorphism \[\Alt(M^{\times r},\ep^{r}M)\to\Alt(\CE_{M}^{\times r},\CE_{\ep^{M}})\cong\Hom(\ep^{r}\CE_{M},\CE_{\ep^{r}M})\]The universal alternating  morphism $M^{\times r}\to\ep^{r}M$ now gives the desired morphism $\mathscr{L}_{M}$.
\end{proof}

\subsection{Multilinear theory of $p$-divisible groups}

In this subsection we recall some constructions and results on multilinear theory of $p$-divisible groups from \cite{H2}, and further develop the theory to study multilinear constructions of the universal covers of $p$-divisible groups and their adic generic fiber. From now on (until the end of the paper), we will assume that $p$ is an odd prime number.\\

\begin{notation}
We denote by $\mathcal{BT}_{h}, \CB\CT_{h,\leq1}, \CB\CT_{h}^{n}$ and $\CB\CT_{h,\leq1}^{n}$ respectively the smooth algebraic stack of $p$-divisible groups of height $h$, $p$-divisible groups of height $h$ and dimension at most $1$, truncated Barsotti-Tate groups of height $h$ and level $n$ and truncated Barsotti-Tate groups of height $h$, level $n$ and dimension at most $1$.\xqed{\star}
\end{notation}

Let us recall the definition of multilinear morphisms of $p$-divisible groups:

\begin{dfn}
Let $S$ be a scheme and $G_{0},\dots,G_{1},\dots,G_{r}$ be $p$-divisible groups over $S$. A \emph{multilinear morphism} $\phi:G_{1}\times\dots\times G_{r}\to G_{0}$ is a system $(\phi_{n})$ of multilinear morphisms of fppf sheaves
\[\phi_{n}:G_{1}[p^{n}]\times\dots\times G_{r}[p^{n}]\to G_{0}[p^{n}]\] compatible with the projections $G_{i}[p^{n+1}]\onto G_{i}[p^{n}]$, in the sense that for all $n$, the following diagram is commutative:
\[\xymatrix{G_{1}[p^{n+1}]\times\dots\times G_{r}[p^{n+1}]\ar[rr]^{\qquad\phi_{n+1}}\ar@{->>}[d]&&G_{0}[p^{n+1}]\ar@{->>}[d]\\G_{1}[p^{n}]\times\dots\times G_{r}[p^{n}]\ar[rr]_{\quad\phi_{n}}&&G_{0}[p^{n}]}\]
\emph{Alternating} and \emph{symmetric} multilinear morphisms are defined similarly.\xqed{\blacktriangle}
\end{dfn}

\begin{lem}
\label{LemIsoMultTate}
Let $S$ be a scheme and $G,G_{0},\dots,G_{1},\dots,G_{r}$ be $p$-divisible groups over $S$. There are canonical homomorphism, functorial in all arguments 
\begin{align}
\label{EqIsoMultTate}
\Mult\big(G_{1}\times\dots\times G_{r},G_{0}\big)&\to \Mult_{\BZ_{p}}\big(T_p(G_{1})\times\dots\times T_p(G_{r}),T_p(G_{0})\big)\\
\label{AltTate}\Alt\big(G^{\times r},G_{0}\big)&\to\Alt_{\BZ_{p}}\big(T_p(G)^{\times r},T_p(G_{0})\big)\\
\Sym\big(G^{\times r},G_{0}\big)&\to\Sym_{\BZ_{p}}\big(T_p(G)^{\times r},T_p(G_{0})\big)
\end{align}
\end{lem}

\begin{proof}
Let $\phi:G_{1}\times\dots\times G_{r}\to G_{0}$ be a multilinear morphism. Taking the inverse limit of $$\phi_{n}:G_{1}[p^{n}]\times\dots\times G_{r}[p^{n}]\to G_{0}[p^{n}]$$ and observing that inverse limit commutes with products, we obtain a multilinear morphism $$T_{p}(\phi):T_{p}(G_{1})\times\dots\times T_{p}(G_{r})\to T_{p}(G_{0})$$
By construction, this homomorphism is functorial in all arguments (as is the Tate module construction). Alternating and symmetric multilinear morphisms are preserved under the homomorphism $T_{p}$ thus defined.
\end{proof}

\begin{thm}
\label{MainThmHed}
Fix natural numbers $1\leq r\leq h$. There exists a unique morphism of stacks
\[\ep^{r}:\CB\CT_{h,\leq 1}^{n}\to \CB\CT_{\binom{h}{r}}^{n}\]
satisfying the following
\begin{enumerate}[(1)]
\item by taking the limit, $\ep^{r}$ induces a morphism
\[\CB\CT_{h,\leq 1}\to \CB\CT_{\binom{h}{r}}\]
\item
if $S$ is a scheme and $G$ is in $\CB\CT_{h,\leq 1}(S)$, then we have a canonical isomorphism 
\begin{align}
\label{ExtPowCommTorPt}
(\ep^{r}G)[p^{n}]\cong \ep^{r}(G[p^{n}])
\end{align}
\item
if $G$ is in $\CB\CT_{h,\leq 1}(S)$, then for any $s\in S$, we have \[\dim (\ep^{r}G)(s)=\binom{h-1}{r-1}\cdot\dim G(s)\]
\item
if $G$ is in $\CB\CT_{h,\leq 1}^{n}(S)$, then $\ep^{r}G$ has the categorical property of exterior powers, i.e., there is an alternating morphism of fppf sheaves \[\lambda_{G}:G^{\times r}\to \ep^{r}G\] that makes $\ep^{r}G$ the $r^{\rm th}$-exterior power of $G$ in the category of finite flat groups schemes over $S$. In particular, for every $S$-scheme $Y$, there is a natural homomorphism \[\lambda_{*}(Y):\ep^{r}\big(G(Y)\big)\to \big(\ep^{r}G\big)(Y)\] of $Y$-valued points.
\item
if $G$ is in $\CB\CT_{h,\leq 1}(S)$, the universal alternating morphisms $$\lambda_{G[p^{n}]}:G[p^{n}]^{\times r} \to \ep^{r}(G[p^{n}])$$ define an alternating morphism 
\begin{align}
\label{UnivAltMorPDivGrp}
\lambda_{G}:G^{\times r}\to \ep^{r}G
\end{align}
that is universal in the category of $p$-divisible groups.
\item
let $G$ be in $\CB\CT_{h,\leq 1}(S)$. The alternating morphism 
\begin{align}
T_p(G)\times\dots\times T_p(G)\to T_{p}(\ep^{r}G)
\end{align}
given by the universal alternating morphism $\lambda_{G}$ and using (\ref{AltTate}) is universal, i.e., it induces an isomorphism of $\BZ_{p}$-sheaves:
\begin{align}
\ep^{r}\big(T_p(G)\big)\cong T_{p}(\ep^{r}G)
\end{align}
\item
if $G$ is in $\CB\CT_{h,\leq 1}^{n}(S)$ and $\alpha:(\BZ/p^{n})^{h}\to G(S)$ is a Drinfeld level structure, then the composition 
\[\ep^{r}\big((\BZ/p^{n})^{h}\big)\arrover{\ep^{r}\alpha}\ep^{r}\big(G(S)\big)\arrover{\lambda_{*}(S)}\big(\ep^{r}G\big)(S)\] (still denoted by $\ep^{r}\alpha$) is a Drinfeld level structure.
\item
if $k$ is a perfect field of characteristic $p$ and $G$ is in $\CB\CT_{h,\leq 1}(\Spec(k))$, then, we have a canonical isomorphism of Dieudonn\'e modules
\begin{align}
\label{IsoExtPowDieudMod}
\BD(\ep^{r}G)\cong \ep^{r}_{W(k)}\big(\BD(G)\big)
\end{align}
\end{enumerate}
\end{thm}

\begin{proof}
This is \cite[Theorem A]{HE}, which relies on \cite[Theorem 3.39 \& Proposition 3.31]{HF} and \cite[Theorem 3.25]{H2}.
\end{proof}


\begin{lem}
\label{LemMutlMorUnivCover}
Let $R$ be a ring on which $p$ is topologically nilpotent. Let $G,G_{0},\dots,G_{r}$ be $p$-divisible groups over $R$. There are canonical homomorphisms, functorial in all arguments:
\begin{align}
\label{EqMutlMorUnivCover}
\Mult\big(G_{1}\times\dots\times G_{r},G_{0}\big)&\to \Mult_{\BQ_{p}}\big(\tilde{G}_{1}\times\dots\times \tilde{G}_{r},\tilde{G}_{0}\big)\\
\label{EqAltUnivCov}\Alt\big(G^{\times r},G_{0}\big)&\to\Alt_{\BQ_{p}}\big(\tilde{G}^{\times r},\tilde{G}_{0}\big)\\
\Sym\big(G^{\times r},G_{0}\big)&\to\Sym_{\BQ_{p}}\big(\tilde{G}^{\times r},\tilde{G}_{0}\big)
\end{align}
\end{lem}

\begin{proof}
This follows from Lemma \ref{LemIsoMultTate}, and the isomorphism of part (3) of Proposition \ref{PropPropUnivBTGrp} (see also (\ref{EqIsomTildGTate})).
\end{proof}

\begin{lem}
\label{LemCommDiagTateUnivCoverMultMor}
Let $R$ be a ring on which $p$ is topologically nilpotent. Let $G,G_{0},\dots,G_{r}$ be $p$-divisible groups over $R$. The following diagrams are commutative:
\[\xymatrix{\Mult\big(G_{1}\times\dots\times G_{r},G_{0}\big)\ar[r]\ar[dr]&\Mult_{\BZ_{p}}\big(T_p(G_{1})\times\dots\times T_p(G_{r}),T_p(G_{0})\big)\ar[d]\\
&\Mult_{\BQ_{p}}\big(\tilde{G}_{1}\times\dots\times \tilde{G}_{r},\tilde{G}_{0}\big)}\]
\[\xymatrix{\Alt\big(G^{\times r},G_{0}\big)\ar[r]\ar[dr]&\Alt_{\BZ_{p}}\big(T_p(G)^{\times r},T_p(G_{0})\big)\ar[d]\\
&\Alt_{\BQ_{p}}\big(\tilde{G}^{\times r},\tilde{G}_{0}\big)}\]
\[\xymatrix{\Sym\big(G^{\times r},G_{0}\big)\ar[r]\ar[dr]&\Sym_{\BZ_{p}}\big(T_p(G)^{\times r},T_p(G_{0})\big)\ar[d]\\
&\Sym_{\BQ_{p}}\big(\tilde{G}^{\times r},\tilde{G}_{0}\big)}\]
where the vertical morphisms are induced by the isomorphism of part (3) of Proposition \ref{PropPropUnivBTGrp}.
\end{lem}

\begin{proof}
This follows from the construction of the oblique morphisms.
\end{proof}

\begin{lem}
Let $\CF_{1},\dots,\CF_{r},\CF$ and $\CG$ be (Zariski, fppf, \'etale etc.) pre-sheaves of Abelian groups on a site, and let us denote by $(\_)^{\sharp}$ the sheafification functor. We have canonical homomorphisms, functorial in all arguments:
\begin{align}
\Mult(\CF_{1}\times\dots\times\CF_{r},\CG)&\to\Mult(\CF^{\sharp}_{1}\times\dots\times\CF_{r}^{\sharp},\CG^{\sharp})\\
\Alt(\CF^{\times r},\CG)&\to\Alt(\CF^{\sharp,\times r},\CG^{\sharp})\\
\Sym(\CF^{\times r},\CG)&\to\Sym(\CF^{\sharp,\times r},\CG^{\sharp})
\end{align}
Furthermore, if $\CG$ is already a sheaf, then the above homomorphisms are bijective.
\end{lem}

\begin{proof}
We will only prove the first statement, as the other two are proven similarly. We will proceed by induction on $r$. If $r=1$, then both statements are true by the construction (universal property) of the sheafification functor. So, assume that we know both statements for $r\geq 1$.\\

Let us denote by $\innHom_{\rm pre}(\CF_{r+1},\CG)$ the pre-sheaf hom group, that is the pre-sheaf that sends $U$ to the group $\Hom_{U}(\CF_{r+1{|_{U}}},\CG_{|_{U}})$. Then, again by the universal property of $(\_)^{\sharp}$, we have a canonical homomorphism 
\begin{align}
\label{EqPreHomSheafify}
\innHom_{\rm pre}(\CF_{r+1},\CG)^{\sharp}\to\innHom(\CF_{r+1}^{\sharp},\CG^{\sharp})
\end{align}which is an isomorphism if $\CG$ is already a sheaf (in fact if $\CG$ is a sheaf, then $\innHom_{\rm pre}(\CF_{r+1},\CG)$ is a sheaf).\\

So, we have homomorphisms
\[\Mult\big(\CF_{1}\times\dots\times\CF_{r+1},\CG\big)\arrover{\cong}\Mult\big(\CF_{1}\times\dots\times\CF_{r},\innHom_{\rm pre}(\CF_{r+1},\CG)\big)\arrover{\rm ind. hyp.}\]
\[\Mult\big(\CF_{1}^{\sharp}\times\dots\times\CF_{r}^{\sharp},\innHom_{\rm pre}(\CF_{r+1},\CG)^{\sharp}\big)\arrover{(\ref{EqPreHomSheafify})}\]\[\Mult\big(\CF_{1}^{\sharp}\times\dots\times\CF_{r}^{\sharp},\Hom(\CF_{r+1}^{\sharp},\CG^{\sharp})\big)\arrover{\cong}\Mult(\CF^{\sharp}_{1}\times\dots\times\CF_{r}^{\sharp},\CG^{\sharp})\]
The composition yields the desired homomorphism. By induction hypothesis, and what we said above, if $\CG$ is already a sheaf, then the above homomorphisms are all isomorphisms.
\end{proof}

\begin{prop}
\label{PropAdGenFibPresMult}
Let $k$ be a perfect field of characteristic $p$ and $(R,R^{+})$  a complete affinoid $\big(W(k)[1/p],W(k)\big)$-algebra. Let $\CF_{1},\dots,\CF_{r},\CF$ and $\CG$ be sheaves on $\text{Nilp}_{R^{+}}^{\rm op}$. Then we have canonical homomorphisms, functorial in all arguments
\begin{align}
\label{EqAdGenFibPresMult}
\Mult(\CF_{1}\times\dots\times\CF_{r},\CG)&\to\Mult(\CF_{1,\eta}^{\rm ad}\times\dots\times\CF_{r,\eta}^{\rm ad},\CG_{\eta}^{\rm ad})\\
\Alt(\CF^{\times r},\CG)&\to\Alt(\CF_{\eta}^{\rm{ad}, \times \mathit{r}},\CG_{\eta}^{\rm ad})\\
\Sym(\CF^{\times r},\CG)&\to\Sym(\CF_{\eta}^{\rm{ad}, \times \mathit{r}},\CG_{\eta}^{\rm ad})
\end{align}
\end{prop}

\begin{proof}
As usual, we will only prove the first statement. Using the above proposition, and by construction of the adic generic fiber functor (Definition \ref{DefAdGenFibFun}), we only have to show that there is a canonical and functorial homomorphism 
\[\Mult(\CF_{1}\times\dots\times\CF_{r},\CG)\to \Mult(\CF^{\flat}_{1}\times\dots\times\CF^{\flat}_{r},\CG^{\flat})\]where here (and only here!) we denote by $\CG^{\flat}$ (and similarly for other terms) the pre-sheaf on $\text{CAff}_{(R,R^{+})}^{\rm op}$ that sends $(S,S^{+})$ to $$\dirlim_{S_{0}\subset S^{+}}\invlim_{n}\CG(S_{0}/p^{n})$$ where the direct limit runs over all open and bounded sub-$R^{+}$-algebras $S_{0}$ of $S^{+}$. So, let $\phi:\CF_{1}\times\dots\times\CF_{r}\to\CG$ be a multilinear morphism. Fix an open and bounded subrings $S_{0}$ of $S^{+}$. For any $n\geq 1$, we have a commutative diagram
\[\xymatrix{\CF_{1}(S_{0}/p^{n+1})\times\dots\times\CF_{r}(S_{0}/p^{n+1})\ar[rr]^{\qquad\phi}\ar[d]&&\CG(S_{0}/p^{n+1})\ar[d]\\
\CF_{1}(S_{0}/p^{n})\times\dots\times\CF_{r}(S_{0}/p^{n})\ar[rr]_{\qquad\phi}&&\CG(S_{0}/p^{n})}\]
and so, we have a multilinear morphism \[\invlim_{n}\CF_{1}(S_{0}/p^{n})\times\dots\times\invlim_{n}\CF_{r}(S_{0}/p^{n})\to\invlim_{n}\CG(S_{0}/p^{n})\]Now, if $S'_{0}\subset S^{+}$ is another open bounded sub-$R^{+}$-algebra containing $S_{0}$, then we have a commutative diagram

\[\xymatrix{\invlim_{n}\CF_{1}(S'_{0}/p^{n})\times\dots\times\invlim_{n}\CF_{r}(S'_{0}/p^{n})\ar[rr]^{\qquad\phi}\ar[d]&&\invlim_{n}\CG(S'_{0}/p^{n})\ar[d]\\
\invlim_{n}\CF_{1}(S_{0}/p^{n})\times\dots\times\invlim_{n}\CF_{r}(S_{0}/p^{n})\ar[rr]_{\qquad\phi}&&\invlim_{n}\CG(S_{0}/p^{n})}\]
Therefore, we have a multilinear morphism
\begin{align}
\label{EqDirInvLimFlat}
\dirlim_{S_{0}\subset S^{+}}\big(\invlim_{n}\CF_{1}(S_{0}/p^{n})\times\dots\times\invlim_{n}\CF_{r}(S_{0}/p^{n})\big)\to \dirlim_{S_{0}\subset S^{+}}\invlim\CG(S_{0}/p^{n})
\end{align}
Since by \cite[Proposition 2.2.2 (i)]{MR3272049} the open bounded sub-$R^{+}$-algebras of $S^{+}$ are filtered, we can distribute the direct limit into the product, i.e., we have a canonical isomorphism 
\[\dirlim_{S_{0}\subset S^{+}}\big(\invlim_{n}\CF_{1}(S_{0}/p^{n})\times\dots\times\invlim_{n}\CF_{r}(S_{0}/p^{n})\big)\cong\] \[\dirlim_{S_{0}\subset S^{+}}\invlim_{n}\CF_{1}(S_{0}/p^{n})\times\dots\times\dirlim_{R_{0}\subset R^{+}}\invlim_{n}\CF_{r}(S_{0}/p^{n})\]
Composing this isomorphism with (\ref{EqDirInvLimFlat}), we obtain the desired multilinear morphism 
\[\CF_{1}^{\flat}(S,S^{+})\times\dots\times\CF_{r}^{\flat}(S,S^{+})\to\CG^{\flat}(S,S^{+})\]
\end{proof}

\begin{cor}
Let $k$ be a perfect field of characteristic $p$, and $(R,R^{+})$ a complete affinoid $\big(W(k)[1/p],W(k)\big)$-algebra. Let $G,G_{0},\dots,G_{r}$ be $p$-divisible groups over some open and bounded subring of $R^{+}$. There are canonical homomorphisms, functorial in all arguments:
\begin{align}
\Mult\big(G_{1}\times\dots\times G_{r},G_{0}\big)&\to \Mult_{\BQ_{p}}\big(T_{p}(G_{1})_{\eta}^{\rm ad}\times\dots\times T_{p}(G_{r})_{\eta}^{\rm ad},T_{p}(G_{0})_{\eta}^{\rm ad}\big)\\
\label{AltGenFibTate}\Alt\big(G^{\times r},G_{0}\big)&\to\Alt_{\BQ_{p}}\big(T_{p}(G)_{\eta}^{\rm{ad}, \times \mathit{r}},T_{p}(G_{0})_{\eta}^{\rm ad}\big)\\
\Sym\big(G^{\times r},G_{0}\big)&\to\Sym_{\BQ_{p}}\big(T_{p}(G)_{\eta}^{\rm{ad}, \times \mathit{r}},T_{p}(G_{0})_{\eta}^{\rm ad}\big)
\end{align}
\end{cor}

\begin{proof}
The first homomorphism is the composition of homomorphisms (\ref{EqIsoMultTate}) and (\ref{EqAdGenFibPresMult}). The other two are given similarly.
\end{proof}

\begin{cor}
Let $k$ be a perfect field of characteristic $p$, and $(R,R^{+})$ a complete affinoid $\big(W(k)[1/p],W(k)\big)$-algebra. Let  $G,G_{0},\dots,G_{r}$ be $p$-divisible groups over some open and bounded subring of $R^{+}$. There are canonical homomorphisms, functorial in all arguments:
\begin{align}
\Mult\big(G_{1}\times\dots\times G_{r},G_{0}\big)&\to \Mult_{\BQ_{p}}\big(\tilde{G}_{1,\eta}^{\rm ad}\times\dots\times \tilde{G}_{r,\eta}^{\rm ad},\tilde{G}_{0,\eta}^{\rm ad}\big)\\
\label{AltUnivCov}\Alt\big(G^{\times r},G_{0}\big)&\to\Alt_{\BQ_{p}}\big(\tilde{G}_{\eta}^{\rm{ad}, \times \mathit{r}},\tilde{G}_{0,\eta}^{\rm ad}\big)\\
\Sym\big(G^{\times r},G_{0}\big)&\to\Sym_{\BQ_{p}}\big(\tilde{G}_{\eta}^{\rm{ad}, \times \mathit{r}},\tilde{G}_{0,\eta}^{\rm ad}\big)
\end{align}
\end{cor}

\begin{proof}
The first homomorphism is the composition of homomorphisms (\ref{EqMutlMorUnivCover}) and (\ref{EqAdGenFibPresMult}). The other two are given similarly.
\end{proof}

\begin{prop}
\label{PropExtPowUnivDrinLev}
Let $k$ be a perfect field of characteristic $p$, and $(R,R^{+})$ a complete affinoid $\big(W(k)[1/p],W(k)\big)$-algebra. Let $G$ a $p$-divisible group of height $h$ and dimension at most $1$ over some open and bounded subring of $R^{+}$. There are canonical homomorphisms 
\begin{align}
\label{ExtPowTateDrinLev}
\mathscr{T}^{r}_{G,(R,R^{+})}:\ep^{r}_{\BZ_{p}}\big(T_{p}({G})_{\eta}^{\rm ad}(R,R^{+})\big)\to T_{p}({\ep^{r}G})_{\eta}^{\rm ad}(R,R^{+})
\end{align}

and

\begin{align}
\label{ExtPowUnivDrinLev}
\mathscr{L}^{r}_{G,(R,R^{+})}:\ep^{r}_{\BQ_{p}}\big(\tilde{G}_{\eta}^{\rm ad}(R,R^{+})\big)\to \widetilde{\ep^{r}G}_{\eta}^{\rm ad}(R,R^{+})
\end{align}
Furthermore, the following diagram, given by the canonical embedding of the Tate module into the universal cover, is commutative
\begin{align}
\label{CommDiagTateUnivCov}
\xymatrix{\ep^{r}\big(T_p(G)^{\rm ad}_{\eta}(R,R^{+})\big)\ar[d]_{\mathscr{T}^{r}_{G,(R,R^{+})}}\ar[r]&\ep^{r}\big(\tilde{G}_{\eta}^{\rm ad}(R,R^{+})\big)\ar[d]^{\mathscr{L}^{r}_{G,(R,R^{+})}}\\
T_{p}(\ep^{r}G)^{\rm ad}_{\eta}(R,R^{+})\ar[r]&\widetilde{\ep^{r}G}_{\eta}^{\rm ad}(R,R^{+})}\end{align}
\end{prop}

\begin{proof}
By Theorem \ref{MainThmHed}, the exterior power $\ep^{r}G$ exists. If in (\ref{AltGenFibTate}) we replace $G_{0}$ with $\ep^{r}G$, the image of the universal alternating morphism $\lambda_{G}:G^{\times r}\to \ep^{r}G$ (\ref{UnivAltMorPDivGrp}) yields an alternating morphism 
\begin{align}
\label{EqUnivAltMorTateAdGenFib}
\big(T_{p}(G)_{\eta}^{\ad}\big)^{\times r}\to T_{p}(\ep^{r}G)_{\eta}^{\ad}
\end{align}
which on evaluating at $(R,R^{+})$ gives the desired homomorphism $\mathscr{T}^{r}_{G,(R,R^{+})}$. Homomorphism $\mathscr{L}^{r}_{G,(R,R^{+})}$ is given similarly, using (\ref{AltUnivCov}).\\

Commutativity of the square follows from Lemma \ref{LemCommDiagTateUnivCoverMultMor} and the functoriality of the homomorphisms in Proposition \ref{PropAdGenFibPresMult}.
\end{proof}

\begin{rem}
\label{RemTateModIsoField}
For a complete affinoid field $(K,K^{+})$ over $(W(k)[1/p],W(k))$, 
\[\mathscr{T}^{r}_{G,(K,K^{+})}:\ep^{r}_{\BZ_{p}}\big(T_{p}({G})_{\eta}^{\rm ad}(K,K^{+})\big)\to T_{p}({\ep^{r}G})_{\eta}^{\rm ad}(K,K^{+})
\]
is an isomorphism.\xqed{\lozenge}
\end{rem}




\subsection{The wedge morphism on the Lubin-Tate tower}

In this subsection, we use exterior powers of $p$-divisible groups and the results from last subsection to construct a morphism from the Lubin-Tate space at infinity to certain Rapoport-Zink spaces at infinity. Fix a perfect field $k$ of characteristic $p$ and a connected $p$-divisible group $H$ over $k$ of dimension 1 and height $h$.\\

Recall the following definition:

\begin{dfn}
Let $\BX$ be a $p$-divisible group over $k$. Define a functor 
$$\Def_{\BX}:\text{Nilp}_{W(k)}\to \Ens$$ by sending $R$ to the set of \emph{deformations} of $\BX$ to $R$, i.e., the set of isomorphism classes of pairs $(G,\rho)$, where $G$ is a $p$-divisible group over $R$ and 
\[\rho:\BX\times_{k}R/p\to G\times_{R}R/p\] is a quasi-isgoney.\xqed{\blacktriangle}
\end{dfn}

Then we have the following theorem of Rapoprt and Zink (\cite[Theorem 3.25]{MR1393439}):

\begin{thm}
The functor $\Def_{\BX}$ is representable by a formal scheme $\CM_{\BX}$ over $\Spf W(k)$, which locally admits a finitely generated ideal of definition.
\end{thm}

Now, recall the definition of the Rapoprt-Zink spaces at infinity:

\begin{dfn}
Let $\BX$ be a $p$-divisible group over $k$, of height $h$. Consider the functor 
$\CM_{\BX}^{\infty}$ on complete affinoid $(W(k)[1/p],W(k))$-algebras, sending $(R,R^{+})$ to the set of isomorphism classes of triples $(G,\rho,\alpha)$, where $(G,\rho)\in \CM_{\BX,\eta}^{\rm ad}(R,R^{+})$ and \[\alpha:\BZ_{p}^{h}\to T_{p}(G)_{\eta}^{\rm ad}(R,R^{}+)\] is a morphism of $\BZ_{p}$-modules such that for all points $x=\Spa(K,K^{+})\in \Spa(R,R^{+})$, the induced morphism 
\[\alpha(x):\BZ_{p}^{h}\to T_{p}(G)_{\eta}^{\rm ad}(K,K^{+})\] is an isomorphism. When $\BX$ has dimension $1$, $\CM_{\BX}^{\infty}$ is called the \emph{Lubin-Tate space at infinity}.\xqed{\blacktriangle}
\end{dfn}

We have the following theorem:

\begin{thm}[\cite{MR3272049}, Theorem 6.3.4.]
The functor $\CM_{\BX}^{\infty}$ is representable by an adic space over $\Spa(W(k)[1/p], W(k))$, and moreover, it is preperfectoid.
\end{thm}

As we said at the beginning, $H$ is a fixed $p$-divisible group over $k$ of dimension $1$ and height $h$.

\begin{cons}
\label{ConsLambdaAdGenFibH}
Applying homomorphism (\ref{AltUnivCov}) to the universal alternating morphism $\lambda_{H}:H^{\times r}\to \ep^{r}H$ given by Theorem \ref{MainThmHed} (5), we obtain an alternating  morphism
\[\lambda_{H,\eta}^{\rm ad}:(\tilde{H}_{\eta}^{\rm{ad}})^{\times \mathit{r}}\to \big(\widetilde{\ep^{r}H}\big)_{\eta}^{\rm ad}\]
Using Notations \ref{NotationLambdaMap}, we then obtain a morphism
\[\Lambda_{r}\lambda_{H,\eta}^{\rm ad}:\big(\tilde{H}_{\eta}^{\rm{ad}}\big)^{\times h}\to\Big(\big(\widetilde{\ep^{r}H}\big)_{\eta}^{\rm{ad}}\Big)^{\times \binom{h}{r}}\]
If $s_{1},\dots,s_{h}$ are sections of $\tilde{H}_{\eta}^{\ad}(R,R^{+})$ and $\vec{s}:=(s_{1},\dots,s_{h})$, we write $\Lambda_{r}\lambda_{H,\eta}^{\rm ad}(\vec{s})=(\ep^{r}_{\mathbf{c}}\vec{s})_{\mathbf{c}\in\binom{h}{r}}$, where we use the following notation:\\

We let $\binom{h}{r}$ denote the set of subsets of $\{1,\dots,h\}$ of size $r$. If $\mathbf{c}=\{c_{1}<\dots<c_{r}$\} (with $1\leq c_{i}\leq h$) is an element of $\binom{h}{r}$, then $\ep^{r}_{\mathbf{c}}\vec{s}$ is the section $\lambda_{H,\eta}^{\ad}(s_{c_{1}},\dots,s_{c_{r}})\in (\widetilde{\ep^{r}H})^{\ad}_{\eta}(R,R^{+})$.\xqed{\blacktriangledown}
\end{cons}

\begin{rem}
Let $(R,R^{+})$ be a perfectoid affinoid $\big(W(k)[1/p],W(k)\big)$-algebra. Using identification (\ref{AdGenFibGlobSec}) and Theorem \ref{MainThmHed} (8), homomorphism 
\[\mathscr{L}^{r}_{H,(R,R^{+})}:\ep^{r}_{\BQ_{p}}\big(\tilde{H}_{\eta}^{\rm ad}(R,R^{+})\big)\to \widetilde{\ep^{r}H}_{\eta}^{\rm ad}(R,R^{+})\]is identified with the homomorphism (\ref{MapFrob=p}):
\[\ep^{r}\Big(\big(\BD(H)\otimes_{W(k)}B^{+}_{\rm cris}(R^{+}/p)\big)^{F=p}\Big)\to \Big(\ep^{r}\big(\BD(H)\otimes_{W(k)}B^{+}_{\rm cris}(R^{+}/p)\big)\Big)^{F=p}\]
and we have a commutative diagram
\begin{align}
\label{RemCommDiagUnivCovHDieudMod}
\xymatrix{\big(\tilde{H}_{\eta}^{\rm{ad}}(R,R^{+})\big)^{\times h}\ar[r]^{\Lambda_{r}\lambda_{H,\eta}^{\rm ad}(R,R^{+})}\ar[d]_{\cong}&\Big(\big(\widetilde{\ep^{r}H}\big)_{\eta}^{\rm{ad}}(R,R^{+})\Big)^{\times \binom{h}{r}}\ar[d]^{\cong}\\
\Big(\big(\BD(H)\otimes_{W(k)}B^{+}_{\rm cris}(R^{+}/p)\big)^{F=p}\Big)^{\times h}\,\,\ar[r]_{\Lambda_{r}\lambda_{\BD(H)}\qquad}^{(\ref{MapExtPowerFrob=p})\qquad}&\Big( \Big(\ep^{r}\big(\BD(H)\otimes_{W(k)}B^{+}_{\rm cris}(R^{+}/p)\big)\Big)^{F=p}\Big)^{\times \binom{h}{r}}}
\end{align}
\xqed{\lozenge}
\end{rem}

\begin{lem}
\label{LemQlogDiagComm}
The following diagram is commutative

\[\xymatrix{\big(\tilde{H}_{\eta}^{\rm ad}(R,R^{+})\big)^{\times h}\ar[d]_{\rm qlog^{\times h}}\ar[rr]^{\Lambda_{r}\lambda_{H,\eta}^{\rm ad}\quad}&&\Big(\big(\widetilde{\ep^{r}H}\big)_{\eta}^{\rm{ad}}(R,R^{+})\Big)^{\times \binom{h}{r}}\ar[d]^{\rm qlog^{\times \binom{h}{r}}}\\\big(\BD(H)\otimes_{W(k)}R\big)^{\times h}\ar[rr]_{\Lambda_{r}\lambda_{\BD(H)}\quad}&&\big(\ep^{r}\BD(H)\otimes_{W(k)}R\big)^{\times \binom{h}{r}}}\]where $\lambda_{\BD(H)}:\BD(H)^{\times r}\to \BD(H)$ is the universal alternating morphism $(x_{1},\dots,x_{r})\mapsto x_{1}\wedge\dots\wedge x_{r}$.
\end{lem}

\begin{proof}
Using commutative diagram (\ref{RemCommDiagUnivCovHDieudMod}) and Lemma \ref{LemCommDiagQLogTheta}, it is enough to show that the following diagram is commutative:
\[\xymatrix{\big(\BD(H)\otimes_{W(k)}B^{+}_{\rm cris}(R^{+}/p)\big)^{\times h}\,\,\ar[r]_{\Lambda_{r}\lambda_{\BD(H)}\qquad}\ar[d]_{\Theta}&\Big(\ep^{r}\big(\BD(H)\otimes_{W(k)}B^{+}_{\rm cris}(R^{+}/p)\big)\Big)^{\times \binom{h}{r}}\ar[d]^{\Theta}\\\big(\BD(H)\otimes_{W(k)}R\big)^{\times h}\,\,\ar[r]_{\Lambda_{r}\lambda_{\BD(H)}\qquad}&\Big(\ep^{r}\big(\BD(H)\otimes_{W(k)}R\big)\Big)^{\times \binom{h}{r}}}\]
The commutativity of this diagram follows from the construction of the horizontal morphisms and the following observation: if $\theta:A\to B$ is a ring homomorphism and $M$ is a free $A$-module of rank $h$, then the following diagram is commutative
\[\xymatrix{M^{\times h}\ar[r]^{\Lambda_{r}\quad}\ar[d]&(\ep_{A}^{r}M)^{\times \binom{h}{r}}\ar[d]\\
(M\otimes_{A}B)^{\times h}\ar[r]_{\Lambda_{r}\quad}&\big(\ep_{B}^{r}(M\otimes_{A}B)\big)^{\times \binom{h}{r}}}\]
where as usual, we are using Notations \ref{NotationLambdaMap} for the horizontal maps (applied to the universal alternating morphisms $(x_{1},\dots,x_{r})\mapsto x_{1}\wedge\dots\wedge x_{r}$).
\end{proof}

\begin{thm}
\label{ThmLambdaRZ}
Let $k$ be a perfect field $k$ of characteristic $p$ and $H$ a $p$-divisible group over $k$ of dimension $1$. Taking exterior powers induces a morphism of adic spaces over $\Spa\big(W(k)[1/p],W(k)\big)$

\[\Lambda^{r}:\CM_{H}^{\infty}\to \CM_{\ep^{r}H}^{\infty}\]
\end{thm}

\begin{proof}
Let $(R,R^{+})$ be a complete affinoid $\big(W(k)[1/p],W(k)\big)$-algebra and $(G,\rho,\alpha)$ an element of $\CM_{H}^{\infty}(R,R^{+})$. Since $G$ is a deformation of $H$, it also has dimension at most $1$ and so, by Theorem \ref{MainThmHed}, $\ep^{r}G$ exists. Set $\Lambda^{r}(G,\rho,\alpha):=(\ep^{r}G,\ep^{r}\rho,\ep^{r}\alpha)$, where \[\ep^{r}\rho:\ep^{r}(H)\times_{k}R/p\cong\ep^{r}(H\times_{k}R/p)\to \ep^{r}(G\times_{R}R/p)\cong\ep^{r}(G)\times_{R}R/p \] is the exterior power of the quasi-isogeny $\rho:H\times_{k}R/p\to G\times_{R}R/p$, which exists thanks to the functoriality of exterior powers, and their base change property. The level structure $\ep^{r}\alpha$ is also given by Theorem \ref{MainThmHed} (7) or as the composition (see (\ref{ExtPowTateDrinLev}))
\[\ep^{r}(\BZ_{p}^{h})\arrover{\ep^{r}\alpha}\ep^{r}\big(T_{p}({G})_{\eta}^{\rm ad}(R,R^{+})\big)\arrover{\mathscr{T}^{r}_{G,(R,R^{+})}} T_{p}({\ep^{r}G})_{\eta}^{\rm ad}(R,R^{+})\]
Over a point $x=\Spa(K,K^{+})\in \Spa(R,R^{+})$  this composition is an isomorphism, since $\alpha$ and $\mathscr{T}^{r}_{G,(K,K^{+})}$ are both isomorphisms (cf. Remark {RemTateModIsoField}).\\
These constructions are functorial and therefore, we obtain the desired morphism 
\[\Lambda^{r}:\CM_{H}^{\infty}\to \CM_{\ep^{r}H}^{\infty}\]
\end{proof}


\begin{notation}
Let $k$ be a perfect field of characteristic $p$ and $G$ a $p$-divisible group over $k$ of height $h$ and dimension $d$. We denote by $\mathscr{F}\ell_{G}$ the flag variety parametrizing $d$-dimensional quotients of the $h$-dimensional $W(k)[1/p]$-vector space $\BD(G)[1/p]$. We consider $\mathscr{F}\ell_{G}$ as an adic space over $\Spa\big(W(k)[1/p],W(k)\big)$.\xqed{\star}
\end{notation}

\begin{cons}
\label{ConsMapFlagVar}
As before, assume that the dimension of $H$ is $1$, so that $\ep^{r}H$ exists. We want to construct a morphism 
\[\FL_{r}:\mathscr{F}\ell_{H}\to\mathscr{F}\ell_{\ep^{r}H}\]
Let $(R,R^{+})$ be a complete affinoid $\big(W(k)[1/p],W(k)\big)$-algebra and $\xi\in\mathscr{F}\ell_{H}(R,R^{+})$ represent the following short exact sequence of $R$-modules \[(\xi)\qquad 0\to K\to \BD(H)\otimes_{W(k)}R\to W\to 0\] where $W$ is a finitely generated projective $R$-module of rank $1$. Since $W$ is projective, $(\xi)$ splits:\[\BD(H)\otimes R\cong K\oplus W\] and since it has rank $1$, we have
\[\ep^{r}\BD(H)\otimes R\cong \ep^{r}K\oplus \ep^{r-1}K\otimes W\]
Therefore, we have the following short exact sequence:
\[0\to\ep^{r}K\to\ep^{r}\BD(H)\otimes_{W(k)}R\to\ep^{r-1}K\otimes_{R}W\to 0\]We define $\FL_{r}(\xi)$ to be this short exact sequence.
\xqed{\blacktriangledown}
\end{cons}

\section{Main Theorem}

In this section, we fix an algebraically closed field $k$ of characteristic $p$ and a $p$-divisible group $H$ over $k$ of height $h$ and dimension $1$.\\

\begin{prop}[\cite{MR3272049}, Lemma 6.3.6]
\label{ThmSchWeinPAdicHdgRZ}
Let $G$ be a $p$-divisible group over $k$ of dimension $d$ and height $h$. The Rapoport-Zink tower $\CM_{G}^{\infty}$ canonically represents the functor on complete affinoid $\big(W(k)[1/p],W(k)\big)$-algebras, sending $(R,R^{+})$ to the set of $h$-tuples
\[(s_{1},\dots,s_{h})\in \big(\tilde{G}^{\rm ad}_{\eta}(R,R^{+})\big)^{\times h}\]satisfying the following conditions:
\begin{enumerate}[(i)]
\item The matrix $\big(\qlog(s_{1}),\dots,\qlog(s_{h})\big)\in \big(\BD(G)\otimes_{W(k)}R\big)^{\times h}\cong\BM_{h}(R)$ is of rank $h-d$. Let $\BD(G)\otimes R\onto W$ be the induced finitely generated projective quotient of rank $d$ (see Lemma \ref{LemRankMatQuot}).
\item
For all geometric points $x=\Spa(\mathbf{C},\CO_{\mathbf{C}})\to \Spa(R,R^{+})$, the sequence
\[0\to\BQ_{p}^{h}\arrover{(s_{1},\dots,s_{h})} \tilde{G}^{\ad}_{\eta}(\mathbf{C},\CO_{\mathbf{C}})\arrover{\qlog} {W}\otimes_{R}\mathbf{C}\to0\]is exact.
\end{enumerate}
Moreover, forgetting conditions (i) and (ii) gives a locally closed embedding $\CM_{G}^{\infty}\subset \big(\tilde{G}^{\ad}_{\eta}\big)^{\times h}$.
\end{prop}
\begin{proof}
We are not going to repeat the proof here, and refer to \cite{MR3272049} for details. We are only going to recall how we obtain such an $h$-tuple from a point on the Rapoprt-Zink tower. Let $(R,R^{+})$ be a complete affinoid $\big(W(k)[1/p],W(k)\big)$-algebra and $(\Gamma,\rho,\alpha)\in\CM^{\infty}_{G}(R,R^{+})$, where $(\Gamma,\rho)$ is defined over some open and bounded subring $R_{0}\subset R^{+}$. The quasi-isogeny $\rho$ provides an identification $\tilde{G}_{R_{0}}\cong \tilde{\Gamma}$, and therefore, we have a morphism
\[\BZ_{p}^{h}\arrover{\alpha}T_{p}(\Gamma)^{\ad}_{\eta}(R,R^{+})\into\tilde{\Gamma}^{\ad}_{\eta}(R,R^{+})\cong\tilde{G}^{\ad}_{\eta}(R,R^{+})\]This morphism provides us with $h$ sections of $\tilde{G}^{\ad}_{\eta}(R,R^{+})$ that satisfy conditions (i) and (ii) above. The rank-$d$ quotient thus obtained (condition (i)) is canonically isomorphic to $\Lie(G)\otimes R$.
\end{proof}

\begin{dfn}
Let us denote by $\leftidx{_{1}}\CM_{G}^{\infty}$ the subsheaf of $\big(\tilde{G}^{\ad}_{\eta}\big)^{\times h}$, whose sections satisfy (only) condition (i) of the above proposition. So, we have inclusions of functors
\[\CM^{\infty}_{G}\subset\leftidx{_{1}}\CM_{G}^{\infty}\subset\big(\tilde{G}^{\ad}_{\eta}\big)^{\times h}\]\xqed{\blacktriangle}
\end{dfn}

\begin{lem}
\label{LemM_1GoesToItself}
The composition $$\leftidx{_{1}}\CM_{H}^{\infty}\into\big(\tilde{H}_{\eta}^{\rm{ad}}\big)^{\times h}\arrover{\Lambda_{r}\lambda_{H,\eta}^{\rm ad}}\Big(\big(\widetilde{\ep^{r}H}\big)_{\eta}^{\rm{ad}}\Big)^{\times \binom{h}{r}}$$ factors through the inclusion 
\[\leftidx{_{1}}\CM_{\ep^{r}H}^{\infty}\into \Big(\big(\widetilde{\ep^{r}H}\big)_{\eta}^{\rm{ad}}\Big)^{\times \binom{h}{r}}\] Furthermore, the resulting square
\[\xymatrix{\leftidx{_{1}}\CM_{H}^{\infty}\ar[rr]\ar[d]&&\leftidx{_{1}}\CM_{\ep^{r}H}^{\infty}\ar[d]\\
\big(\tilde{H}_{\eta}^{\rm{ad}}\big)^{\times h}\ar[rr]_{\Lambda_{r}\lambda_{H,\eta}^{\rm ad}}&&\Big(\big(\widetilde{\ep^{r}H}\big)_{\eta}^{\rm{ad}}\Big)^{\times \binom{h}{r}}}\]
is Cartesian.
\end{lem}

\begin{proof}
Let $(R,R^{+})$ be an affinoid $\big(W(k)[1/p],W(k)\big)$-algebra. Let $s_{1},\dots,s_{h}$ be sections of $\tilde{H}_{\eta}^{\ad}(R,R^{+})$ and set $\vec{s}:=(s_{1},\dots,s_{h})$. Let us denote by $Q_{\vec{s}}$ the matrix $\big(\qlog(s_{1}),\dots,\qlog(s_{h})\big)\in\big(\BD(H)\otimes_{W(k)}R\big)^{\times h}\cong\BM_{h}(R)$. Similarly, let us denote by $Q_{\Lambda\vec{s}}\in\BM_{\binom{h}{r}}(R)$ the matrix obtained, by applying $\qlog$, from $\Lambda_{r}(\vec{s})=(\ep^{r}_{\mathbf{c}}\vec{s})_{\mathbf{c}\in\binom{h}{r}}$  (see Construction \ref{ConsLambdaAdGenFibH} for notations), i.e., $Q_{\Lambda\vec{s}}=\big(\qlog(\ep^{r}_{\mathbf{c}}\vec{s})\big)$. By Lemma \ref{LemQlogDiagComm}, we have $Q_{\Lambda\vec{s}}=\ep^{r}Q_{\vec{s}}$ (we use Notations \ref{NotExtPowMat}). It follows from Lemma \ref{RankGivenByExtPow} that $Q_{\vec{s}}$ has rank $h-1$ if and only if $Q_{\Lambda\vec{s}}$ has rank $\binom{h-1}{r}=\binom{h}{r}-\binom{h-1}{r-1}$. This achieves the proof.
\end{proof}

The following lemma describes the morphism $\leftidx{_{1}}\CM_{H}^{\infty}\to \leftidx{_{1}}\CM_{\ep^{r}H}^{\infty}$ explicitly:

\begin{lem}
\label{LemExplExtShortExSeq}
Let $(R,R^{+})$ be an affinoid $\big(W(k)[1/p],W(k)\big)$-algebra and let $\vec{s}:=(s_{1},\dots,s_{h})$ be an element of $\leftidx{_{1}}\CM_{H}^{\infty}(R,R^{+})$ defining the short exact sequence:
\begin{align}
\label{EqShExSeqDieuMod}
\qquad0\to K\to \BD(H)\otimes_{W(k)}R\to W\to 0
\end{align} in other words, $W$ is the rank-$1$ finitely generated projective module, given as the quotient of $\BD(H)\otimes R$ by the submodule $K$, generated by $\qlog(s_{1}),\dots,\qlog(s_{h})$. The short exact sequence defined by $\Lambda_{r}(\vec{s})$ is 
\[0\to \ep^{r}K\to \ep^{r}\BD(H)\otimes R\to \ep^{r-1}K\otimes W\to 0\]
\end{lem}

\begin{proof}
This follows from Lemma \ref{LemQlogDiagComm} and the fact that sequence (\ref{EqShExSeqDieuMod}) splits (see also Construction \ref{ConsMapFlagVar}).
\end{proof}

\begin{prop}
\label{PropFirstCartDiag}
The following diagram is Cartesian:
\begin{align}
\label{CartDiagM-1M}
\xymatrix{\CM_{H}^{\infty}\ar[d]\ar[rr]^{\Lambda^{r}\quad}&&\CM_{\ep^{r}H}^{\infty}\ar[d]\\
\leftidx{_{1}}\CM_{H}^{\infty}\ar[rr]_{\Lambda_{r}}&&\leftidx{_{1}}\CM_{\ep^{r}H}^{\infty}}
\end{align}
\end{prop}

\begin{proof}
Let $(R,R^{+})$ be an affinoid $\big(W(k)[1/p],W(k)\big)$-algebra. Let us first show that the diagram is commutative. Let $\vec{s}:=(s_{1},\dots,s_{h})$ be an $h$-tuple of sections of $\tilde{H}_{\eta}^{\ad}(R,R^{+})$, belonging to $\CM_{H}^{\infty}(R,R^{+})$.	By Proposition \ref{ThmSchWeinPAdicHdgRZ}, there is a triple $(\Gamma, \rho,\alpha)\in \CM_{H}^{\infty}(R,R^{+})$, such that $s_{i}$ are given by the morphism 
\[\BZ_{p}^{h}\arrover{\alpha}T_{p}(\Gamma)^{\ad}_{\eta}(R,R^{+})\into\tilde{\Gamma}^{\ad}_{\eta}(R,R^{+})\cong\tilde{H}^{\ad}_{\eta}(R,R^{+})\]Let us chase the element $\vec{s}$ in this diagram. Under $\Lambda^{r}$ (the top horizontal morphism) it goes to the sections representing $(\ep^{r}\Gamma,\ep^{r}\rho,\ep^{r}\alpha)$ (see Theorem \ref{ThmLambdaRZ}). More precisely, it goes to the section given by the composition
\begin{align*}
\ep^{r}(\BZ_{p}^{h})&\arrover{\ep^{r}\alpha}\ep^{r}\big(T_{p}({\Gamma})_{\eta}^{\rm ad}(R,R^{+})\big)\arrover{\mathscr{T}^{r}_{\Gamma,(R,R^{+})}}\\
&T_{p}({\ep^{r}\Gamma})_{\eta}^{\rm ad}(R,R^{+})\into(\widetilde{\ep^{r}\Gamma})^{\ad}_{\eta}(R,R^{+})\cong(\widetilde{\ep^{r}H})^{\ad}_{\eta}(R,R^{+})
\end{align*}
Under $\Lambda_{r}$ (the bottom horizontal morphism), $\vec{s}$ goes to the section given by the composition 
\begin{align*}
\ep^{r}(\BZ_{p})&\arrover{\ep^{r}\alpha}\ep^{r}\big(T_{p}({\Gamma})_{\eta}^{\rm ad}(R,R^{+})\big)\to\\
&\ep^{r}\big(\widetilde{\Gamma}^{\ad}_{\eta}(R,R^{+})\big)\cong\ep^{r}\big(\widetilde{H}^{\ad}_{\eta}(R,R^{+})\big)\arrover{\mathscr{L}^{r}_{H,(R,R^{+})}}(\widetilde{\ep^{r}H})^{\ad}_{\eta}(R,R^{+})
\end{align*}
Proposition \ref{PropExtPowUnivDrinLev} (diagram (\ref{CommDiagTateUnivCov})) states that these two compositions are equal. It implies that the diagram is commutative.\\

Let us now show that the diagram is Cartesian. So, pick $\vec{s}=(s_{r},\dots,s_{h})$ in $\leftidx{_{1}}\CM_{H}^{\infty}(R,R^{+})$ such that $\Lambda_{r}(\vec{s})=(\ep^{r}_{\mathbf{c}}\vec{s})_{\mathbf{c}\in\binom{h}{r}}$ is in $\CM_{\ep^{r}H}^{\infty}(R,R^{+})$. Let 
\[0\to K\to \BD(H)\otimes_{W(k)}R\to W\to 0\] be the short exact sequence defined by $\vec{s}$. In order to show that $\vec{s}$ belongs to $\CM_{H}^{\infty}$, we have to show that for all geometric generic points $x=\Spa(\mathbf{C},\CO_{\mathbf{C}})\to\Spa(R,R^{+})$, the sequence 
\[0\to\BQ_{p}^{h}\arrover{\vec{s}} \tilde{H}^{\ad}_{\eta}(\mathbf{C},\CO_{\mathbf{C}})\arrover{\qlog} {W}\otimes_{R}\mathbf{C}\to0\]is exact.\\

By Lemma \ref{LemExplExtShortExSeq} the short exact sequence defined by $\Lambda_{r}(\vec{s})$ is
\[0\to \ep^{r}K\to \ep^{r}\BD(H)\otimes R\to \ep^{r-1}K\otimes W\to 0\]\\


Set $\CF:=\CO_{X}^{h}$. By Lemma \ref{LemGlobSecVectUnivCov}, sections $s_{1},\dots,s_{h}$ define a morphism $\vec{s}:\CF\to \CE_{H}$ of vector bundles on the Fargues-Fontaine curve $X$ (here $\CE_{H}$ is the vector bundle associated with $H$, i.e., $\CE_{\BD(H)\otimes_{W(k)}B^{+}_{\rm cris}(\CO_{\mathbf{C}}/p)}$). Similarly, sections $\ep^{r}_{\mathbf{c}}\vec{s}$ define a morphism $\Lambda_{r}\vec{s}:\ep^{r}\CF\to\CE_{\ep^{r}H}$. Note that we have a commutative diagram
\begin{align}
\label{TriangleFFCurve}
\xymatrix{\ep^{r}\CF\ar[r]^{\ep^{r}\vec{s}}\ar[dr]_{\Lambda_{r}(\vec{s})}&\ep^{r}\CE_{H}\ar[d]^{\mathscr{L}}\\&\CE_{\ep^{r}H}}
\end{align}
By \cite[Theorem 6.2.1]{MR3272049}, the sequence 
\[0\to\ep^{r}\CF\to \CE_{\ep^{r}H}\to i_{\infty*}(\ep^{r-1}K\otimes W\otimes\mathbf{C})\to0\] is exact. In particular, the oblique morphism in diagram (\ref{TriangleFFCurve}) is a monomorphism, which implies that $\ep^{r}\vec{s}:\ep^{r}\CF\to\ep^{r}\CE_{H}$ is a monomorphism as well. It follows that $\CF\arrover{\vec{s}} \CE_{H}$ is a monomorphism. Let $\CV$ be the cokernel of $\CF\arrover{\vec{s}} \CE_{H}$.\\

Since $\ep^{r}\CF\arrover{\Lambda_{r}(\vec{s})}\CE_{\ep^{r}H}$ is an isomorphism away from $\infty\in X$, and both $\ep^{r}\CE_{H}$ and $\CE_{\ep^{r}H}$ have rank $\binom{h}{r}$,  $\mathscr{L}:\ep^{r}\CE_{H}\to\CE_{\ep^{r}H}$ is an isomorphism away from $\infty$. It follows that $\ep^{r}\CF\arrover{\ep^{r}\vec{s}}\ep^{r}\CE_{H}$ is an isomorphism away from $\infty$ as well. By \cite[Corollary 2.3]{H6}, $\CF\arrover{\vec{s}}\CE_{H}$ has the same property and we have a modification of vector bundles \[0\to\CF\arrover{\vec{s}}\CE_{H}\to i_{\infty*}U\to0\]  where $U$ is a $\mathbf{C}$-vector space. Counting degrees of the members of this sequence, we have $\dim U=1$.
Note that by construction, the composition $\CF\arrover{\vec{s}}\CE_{H}\to i_{\infty *}(W\otimes \mathbf{C})$ is zero and $\CE_{H}\to i_{\infty*}(W\otimes\mathbf{C})$ is an epimorphism. Therefore, we have a commutative diagram
\begin{align*}
\xymatrix{0\ar[r]&\CF\ar[r]^{\vec{s}}&\CE_{H}\ar[r]\ar@{->>}[d]&i_{\infty *}U\ar[r]\ar@{->>}[dl]&0\\
&&i_{\infty *}(W\otimes \mathbf{C})&&}
\end{align*}
As $U$ and $W\otimes\mathbf{C}$ have both dimension $1$, $i_{\infty *}U\to i_{\infty *}(W\otimes \mathbf{C})$ is in fact an isomorphism and so, the following sequence is exact:
\[0\to\CF\arrover{\vec{s}}\CE_{H}\to i_{\infty *}(W\otimes \mathbf{C})\to0\]
Taking global sections, we obtain the short exact sequence:
\[0\to\BQ_{p}^{h}\arrover{\vec{s}} \tilde{H}^{\ad}_{\eta}(\mathbf{C},\CO_{\mathbf{C}})\arrover{\qlog} {W}\otimes_{R}\mathbf{C}\to0\]as desired.
\end{proof}

\begin{prop}
The morphism $\mathscr{L}:\ep^{r}\CE_{H}\to\CE_{\ep^{r}H}$ is an isomorphism.
\end{prop}

\begin{proof}
We saw in the proof that away from $\infty$ this morphism is an isomorphism. Note that we proved this via the ``auxiliary'' vector bundle $\CF$ (diagram (\ref{TriangleFFCurve})) , which can be constructed by taking any $\Spa(\mathbf{C},\CO_{\mathbf{C}})$-point of $\CM_{H}^{\infty}$. So, we only need to show that it is as isomorphism at $\infty$ as well. Since $\ep^{r}\CE_{H}$ and $\CE_{\ep^{r}H}$ are vectors bundles of the same rank, it is enough to show that $\mathscr{L}_{\infty}$ is an epimorphism. By Nakayama's lemma, it is then enough to show that $$i_{\infty}^{*}\mathscr{L}:i_{\infty}^{*}\ep^{r}\CE_{H}\to i_{\infty}^{*}\CE_{\ep^{r}H}$$ in an epimorphism. We have
\[i_{\infty}^{*}\ep^{r}\CE_{H}\cong\ep^{r}(i_{\infty}^{*}\CE_{H})\cong \ep^{r}(\BD(H)\otimes_{W(k)}\mathbf{C})\] and 
\[i_{\infty}^{*}\CE_{\ep^{r}H}\cong\BD(\ep^{r}H)\otimes_{W(k)}\mathbf{C}\] and the morphism $i_{\infty}^{*}\mathscr{L}$ is nothing but the isomorphism (\ref{IsoExtPowDieudMod}) of Theorem \ref{MainThmHed} tensored with $\mathbf{C}$.
\end{proof}

\begin{thm}
The following diagram is Cartesian
\[\xymatrix{\CM_{H}^{\infty}\ar[d]\ar[rr]^{\Lambda^{r}\quad}&&\CM_{\ep^{r}H}^{\infty}\ar[d]\\\big(\tilde{H}_{\eta}^{\rm{ad}}\big)^{\times h}\ar[rr]_{\Lambda_{r}\lambda_{H,\eta}^{\rm ad}\quad}&&\Big(\big(\widetilde{\ep^{r}H}\big)_{\eta}^{\rm{ad}}\Big)^{\times \binom{h}{r}}}\]
\end{thm}

\begin{proof}
If $r=h$, then this is \cite[Theorem 6.4.1]{MR3272049}. Assume $r<h$. The statement follows immediately from Lemma \ref{LemM_1GoesToItself} and Proposition \ref{PropFirstCartDiag}.
\end{proof}

Let $k$ be a prefect field of characteristic $p$ and $G$ a $p$-divisible group over $k$. This proves the following lemma:

\begin{lem}
\label{LemPerMorFact}
Let $k$ be a prefect field of characteristic $p$ and $G$ a $p$-divisible group over $k$. The peroid morphism $\pi_{G}:\CM_{G}^{\infty}\to\mathscr{F}\ell_{G}$ factors through $\CM_{G}^{\infty}\into\leftidx{_{1}}\CM_{G}^{\infty}$.
\end{lem}
\begin{proof}
Recall that the period morphism is defined, using Grothendieck-Messing deformation theory, by sending a deformation $(G',\rho,\alpha)$ (up to quasi-isogeny) over $R$ to the quotient $$\BD(G)\otimes_{W(k)}R[1/p]\onto\Lie(G')[1/p]$$ It follows from what we said in the proof of Proposition \ref{ThmSchWeinPAdicHdgRZ} (regarding $\CM_{G}^{\infty}$ as a subsheaf of $\big(\tilde{G}_{\eta}^{\rm{ad}}\big)^{\times h}$) that the peroid morphism $\pi_{G}$ factors through $\CM_{G}^{\infty}\into\leftidx{_{1}}\CM_{G}^{\infty}$. Therefore, if $\vec{s}\in \leftidx{_{1}}\CM_{G}^{\infty}(R,R^{+})$ define the short exact sequence \[0\to K\to \BD(G)\otimes_{W(k)}R\to W\to 0\] then the image $\vec{s}$ under $\pi_{G}$ is the quotient $\BD(G)\otimes_{W(k)}R\to W\to 0$. 
\end{proof}

\begin{lem}
\label{LemFirsPerMorDiag}
The following diagram is commutative
\[\xymatrix{\leftidx{_{1}}\CM_{H}^{\infty}\ar[d]_{\pi_{H}}\ar[rr]^{{\Lambda_{r}}}&&\leftidx{_{1}}\CM_{\ep^{r}H}^{\infty}\ar[d]^{\pi_{\ep^{r}H}}\\\mathscr{F}\ell_{H}\ar[rr]_{\FL_{r}}&&\mathscr{F}\ell_{\ep^{r}H}}\]
\end{lem}

\begin{proof}
This follows from Lemma \ref{LemExplExtShortExSeq}, the proof of Lemma \ref{LemPerMorFact} and the construction of $\FL_{r}$ (Construction \ref{ConsMapFlagVar}).
\end{proof}

\begin{prop}
The following diagram is commutative
\[\xymatrix{\CM_{H}^{\infty}\ar[d]_{\pi_{H}}\ar[rr]^{{\Lambda^{r}}}&&\CM_{\ep^{r}H}^{\infty}\ar[d]^{\pi_{\ep^{r}H}}\\\mathscr{F}\ell_{H}\ar[rr]_{\FL_{r}}&&\mathscr{F}\ell_{\ep^{r}H}}\]
\end{prop}

\begin{proof}
This follows from Lemma \ref{LemPerMorFact}, Lemma \ref{LemFirsPerMorDiag} and Proposition \ref{PropFirstCartDiag}.
\end{proof}

\newpage

\phantomsection
\addcontentsline{toc}{section}{References}

\footnotesize{

\bibliographystyle{acm}}


\end{document}